\documentclass[a4paper,12pt]{amsart}
\usepackage[utf8]{inputenc}
\usepackage{amsmath}
\usepackage{amsfonts}
\usepackage{amsthm}
\usepackage{amssymb}
\usepackage[mathscr]{euscript}
\usepackage{xcolor}
\usepackage{enumerate}
\usepackage{tikz}
\usepackage{mathtools}

\usepackage[draft=false,hidelinks]{hyperref}
\usepackage[noabbrev, capitalise, nameinlink]{cleveref} 

\usepackage{scalerel,stackengine}
\stackMath
\newcommand\reallywidehat[1]{%
\savestack{\tmpbox}{\stretchto{%
  \scaleto{%
    \scalerel*[\widthof{\ensuremath{#1}}]{\kern-.6pt\bigwedge\kern-.6pt}%
    {\rule[-\textheight/2]{1ex}{\textheight}}
  }{\textheight}%
}{0.5ex}}%
\stackon[1pt]{#1}{\tmpbox}%
}

\usepackage[
margin=1in,
marginpar=2cm,
includefoot,
footskip=30pt,
]{geometry}

\newtheorem{theorem}[equation]{Theorem}
\newtheorem{lemma}[equation]{Lemma}
\newtheorem{proposition}[equation]{Proposition}
\newtheorem{corollary}[equation]{Corollary}

\newtheorem{notation}[equation]{Notation}

\theoremstyle{definition}
\newtheorem{definition}[equation]{Definition}

\theoremstyle{remark}
\newtheorem{example}[equation]{Example}
\newtheorem{remark}[equation]{Remark}

\numberwithin{equation}{section}

\newcommand{\osf}{{\normalfont \textsf{X}}}

\newcommand{\lang}{\CL_{\osf}}

\newcommand{\ualgshift}{\TCA_R(\osf)}

\newcommand{\udalgshift}{\TCD_R(\osf)}

\newcommand{\ucalgshift}{\TCO_{\osf}}

\newcommand{\alf}{\mathscr{A}}

\newcommand{\N}{\mathbb{N}}
\newcommand{\F}{\mathbb{F}}

\newcommand{\CA}{\mathcal{A}}

\newcommand{\CD}{\mathcal{D}}

\newcommand{\CL}{\mathcal{L}}
\newcommand{\CO}{\mathcal{O}}

\newcommand{\TCA}{\widetilde{\CA}}
\newcommand{\TCB}{\mathcal{U}}
\newcommand{\TCD}{\widetilde{\CD}}
\newcommand{\TCO}{\widetilde{\CO}}

\newcommand{\HTCB}{\widehat{\TCB}}

\newcommand{\tauh}{\widehat{\tau}}
\newcommand{\varphih}{\widehat{\varphi}}

\newcommand{\nn}{\mathbb{N}}

\newcommand{\scj}{\subseteq}

\newcommand{\eword}{\omega}

\newcommand{\Lc}{\operatorname{Lc}}

\newcommand{\Orb}{\operatorname{Orb}}

\title[Partial group algebras with relations]{The dynamical structure of partial group algebras with relations, with applications to subshift algebras}
\author[G. Boava]{Giuliano Boava}
\author[G.G. de Castro]{Gilles G. de Castro}
\author[D. Gonçalves]{Daniel Gonçalves}
\address[Giuliano Boava, Gilles G. de Castro and Daniel Gonçalves]{Departamento de Matem\'atica, Universidade Federal de Santa Catarina, 88040-970 Florian\'opolis SC, Brazil. }
\email{g.boava@ufsc.br \\ gilles.castro@ufsc.br \\ daemig@gmail.com}
\author[D.W. van Wyk]{Daniel W. van Wyk}
\address[Danieldvanwyk@fairfield.edu]{Department of Mathematics, Fairfield University, Fairfield, CT 06824 USA.}
\email{dvanwyk@fairfield.edu }


\begin{document}

\begin{abstract}

We introduce partial group algebras with relations in a purely algebraic framework. Given a group and a set of relations, we define an algebraic partial action and prove that the resulting partial skew group ring is isomorphic to the associated partial group algebra with relations. Under suitable conditions - which always holds if the base ring is a field - we demonstrate that the partial skew group ring can also be described using a topological partial action. Furthermore, we show how subshift algebras can be realized as partial group algebras with relations. Using the topological partial action, we describe simplicity of subshift algebras in terms of the underlying dynamics of the subshift.
\end{abstract}

\maketitle

\section{Introduction}

The realization of C*-algebras generated by partial isometries and projections as partial group algebras with relations—and consequently as partial crossed products—has proven to be a powerful tool in the study of C*-algebras and dynamical systems. Such realizations introduce standard gradings in these algebras, provide criteria for simplicity, and offer an intrinsic description of the dynamics encoded in the algebras (see \cite{ExelLacaQuigg, Royer}). For instance, this technique has been employed to study Cuntz-Krieger algebras for arbitrary infinite matrices (see \cite{ExelLacaQuigg}) and Cuntz-Li algebras (see \cite{BoavaExel}). 

The interplay between algebraic and analytical methods in capturing analytic and dynamical aspects in the noncommutative world has been explored with great success, for instance, in advances in problems such as the Williams problem (see \cite{COS, CGGH, MR4808570}). If one searches the literature for purely algebraic partial group algebras, one is typically directed to C*-algebraic papers or to \cite{MEP}. For C*-algebras, partial  group algebras with relations are studied in \cite{ExelLacaQuigg}. In \cite{MEP}, however,  partial group algebras are defined purely algebraically but without relations. In this paper, we study partial  group algebras with relations, aiming to address this gap in the purely algebraic setting. This generalizes the definition of partial group algebras given in \cite{MEP, MR3818287} and provides an algebraic analogue to the construction in \cite{ExelLacaQuigg} for C*-algebras.

In our approach, we define the partial group algebra with relations as the universal algebra with respect to partial representations of the group that adheres to the given set of relations. A pivotal aspect of our defined algebras is their realization as partial skew group rings. This is accomplished in a purely algebraic manner in Section~2. Specifically, in Proposition~\ref{quefome}, we show that a partial group algebra with relations can be realized as a partial skew group ring of an action on a commutative algebra.

Aiming for a topological description of the partial skew ring above, in Section~3, we use Keimel’s result (\cite[Theorem 1]{Keimel}) to prove that a universal algebra generated by commuting idempotents is isomorphic to the algebra of locally constant functions of a Stone space (Proposition~\ref{almoco}). We also establish results that allow us to view the commutative algebra in the partial skew group ring of Proposition~\ref{quefome} as the algebra of locally constant functions over a certain topological space (Proposition~\ref{quotient.not.field}). 

In Section~4, we describe a partial action on the aforementioned topological space and demonstrate that the induced algebraic partial action is equivariant to the algebraic partial action of Section~2. This provides a description of a partial group algebra with relations as a partial skew group ring arising from a topological partial action (see Theorem~\ref{grevedosalunos}).

In Section~5, we apply our results to subshift algebras. Subshift algebras are defined in both the algebraic and analytic settings, see \cite{BCGWCStarSubshift, BCGWSubshift, BrixCarlsen,CarlsenSubshift}, and they encode the dynamics of subshifts, even over infinite alphabets. Further properties of these algebras and their connections with infinite alphabet subshifts have been studied in \cite{MR4721165, DR19, Gonçalves_Royer_2024}. However, in the case of infinite alphabets, simplicity criteria for such algebras remain unresolved. Using our techniques, we address simplicity. First, we demonstrate that a subshift algebra can be viewed as a partial group algebra with relations (Theorem~\ref{clouds}). Consequently, we obtain a topological partial action such that the associated partial skew group ring is isomorphic to the subshift algebra. We show that this topological action is equivalent to the one found in \cite{BCGWSubshift}, and with this new description we characterize minimality of the action in terms of properties of the subshift. Together with the description of topological freeness given in \cite{BCGWSubshift}, this allows us to provide complete criteria for simplicity of subshift algebras (Theorem~\ref{simple}). Furthermore, since the partial actions and their transformation groupoids used in both the analytical and algebraic contexts are the same, our simplicity criteria also applies to C*-algebras of subshifts (Theorem~\ref{simplecstar}). Finally, we note that, in the finite alphabet case, our conditions can be shown to be equivalent to the criteria established in \cite{MishaRuy} for C*-algebras associated with subshifts over finite alphabets.

\section{Partial Group Algebras}\label{greve}

In this section, we revisit the definition of a partial representation and a partial group algebra, as outlined in \cite{DokEx05, MEP}. Following this, we introduce a new variant of the partial group algebra by incorporating additional relations among certain generators, as discussed in the C*-algebraic context in \cite{ExelLacaQuigg}. The main result of this section is a characterization of the partial group algebra with these relations as a partial skew group ring.

Throughout this section, \( R \) denotes a nonzero unital commutative ring. 

\begin{definition}(\cite[Definition~6.1]{DokEx05})
 \label{partialrep}
Let $G$ be a group. A \textit{partial representation of $G$ in a unital $R$-algebra $B$} is a map $\pi:G\to B$ such that
\begin{enumerate}[(i)]
    \item $\pi(e) = 1$;
    \item $\pi(g)\pi(h)\pi(h^{-1}) = \pi(gh)\pi(h^{-1})$ for all $g,h\in G$;
    \item $\pi(g^{-1})\pi(g)\pi(h) = \pi(g^{-1})\pi(gh)$ for all $g,h\in G$.
\end{enumerate}
\end{definition}

\begin{definition}(\cite[Definition~2.4]{ MEP}, \cite{DokEx05})\label{defi.partial.algebra}
    Let $G$ be a group. The \emph{partial group algebra of $G$ with coefficients in $R$}, denoted by $R_{\operatorname{par}}(G)$, is the universal $R$-algebra with generators $\{[g] : g\in G\}$ subject to the relations
\begin{enumerate}[(i)]
    \item $[e] = 1$;
    \item $[g^{-1}][g][h] = [g^{-1}][gh]$;
    \item $[g][h][h^{-1}] = [gh][h^{-1}]$.
\end{enumerate}
\end{definition}

After revisiting the definition of the partial group algebra above, we now introduce our proposed definition of the group algebra with extra relations. This definition is inspired by the work in \cite{ExelLacaQuigg} within the context of C*-algebras.

\begin{definition} Let $G$ be a group. Elements of the form $[g][g^{-1}]$ are denoted by $\varepsilon_g$. If $\mathcal{R}$ is a set of relations on the set $\{\varepsilon_g:g\in G\}$, we define the \emph{partial group algebra of $G$ with coefficients in $R$ and relations $\mathcal{R}$}, denoted by $R_{\operatorname{par}}(G, \mathcal{R})$, as the universal $R$-algebra with generators $\{[g]: g\in G\}$ subject to the relations (i), (ii) and (iii) in Definition~\ref{defi.partial.algebra}, and the relations in $\mathcal{R}$.  
\end{definition}

\begin{remark}
    Each relation in \(\mathcal{R}\) can be expressed in the form \(\sum_i\lambda_i\prod_j\varepsilon_{g_{ij}}=0\), where \(\lambda_i \in R\) and \(g_{ij} \in G\). For simplicity, when we say that an element \(\sum_i\lambda_i\prod_j\varepsilon_{g_{ij}}\) belongs to \(\mathcal{R}\), it should be understood that the corresponding relation \(\sum_i\lambda_i\prod_j\varepsilon_{g_{ij}}=0\) is included in \(\mathcal{R}\).
\end{remark}

\begin{remark}
    The partial group algebra $R_{\operatorname{par}}(G)$ is clearly universal with respect to partial representations of $G$. An analogous property holds for $R_{\operatorname{par}}(G, \mathcal{R})$, which is universal with respect to partial representations of $G$ that satisfy $\mathcal{R}$, that is, partial representations $\pi$ of $G$ such that \(\sum_i\lambda_i\prod_j\pi(g_{ij})\pi(g_{ij}^{-1})=0\) for each relation \(\sum_i\lambda_i\prod_j\varepsilon_{g_{ij}}=0\) in $\mathcal{R}$.
\end{remark}

\begin{notation} We denote by $A_{\operatorname{par}}(G)$ the $R$-subalgebra of $R_{\operatorname{par}}(G)$ generated by the set $\{\varepsilon_g : g\in G\}$. Similarly, $A_{\operatorname{par}}(G,\mathcal{R})$ is the $R$-subalgebra of $R_{\operatorname{par}}(G, \mathcal{R})$ generated by $\{\varepsilon_g : g\in G\}$.
\end{notation}

The following lemma is \cite[Proposition~9.8]{ExelBook} rephrased in the context of partial group algebras with relations.

\begin{lemma}\label{semsentido}
    Let $G$ be a group and $\mathcal{R}$ a set of relations on the set $\{\varepsilon_g:g\in G\}$. Then
    \begin{enumerate}[(i)]
        \item $[g][g^{-1}][g]=\varepsilon_{g}[g]=[g]$ for all $g\in G$,
        \item $\varepsilon_g^2=\varepsilon_g$,
        \item $\varepsilon_g\varepsilon_h=\varepsilon_h\varepsilon_g$ for all $g,h\in G$,
        \item $[g]\varepsilon_h=\varepsilon_{gh}[g]$ for all $g,h\in G$,
        \item $[g_1]\cdots[g_n]=\varepsilon_{g_1}\varepsilon_{g_1g_2}\cdots\varepsilon_{g_1\cdots g_{n-1}}[g_1\cdots g_n]$ for all $g_1,\ldots,g_n\in G$.  
    \end{enumerate}
\end{lemma}

Note that since the \(\varepsilon_g\)'s commute, both \(A_{\operatorname{par}}(G)\) and \(A_{\operatorname{par}}(G,\mathcal{R})\) are commutative \(R\)-algebras. Furthermore, it is straightforward to see that \(R_{\operatorname{par}}(G)\) is universal with respect to partial representations, and similarly, \(R_{\operatorname{par}}(G, \mathcal{R})\) is universal with respect to partial representations that satisfy the relations in \(\mathcal{R}\).

Next, we recall the realization of the group algebra as a partial skew group ring (see \cite{ExelBook} for the details regarding the theory of partial skew group rings), and of \(A_{\operatorname{par}}(G)\) as a universal algebra generated by elements and relations, as done in \cite{DokEx05} and \cite{ExelBook}. 

\begin{proposition}\label{uni.prop.apar}
    Let $G$ be a group. For each $g\in G$, define $D_g = A_{\operatorname{par}}(G)\varepsilon_g$ and let $\theta_g:D_{g^{-1}}\to D_g$ be given by $\theta_g(\varepsilon_h\varepsilon_{g^{-1}}) = \varepsilon_{gh}\varepsilon_g$. Then, $\theta$ defines a partial action of $G$ on $A_{\operatorname{par}}(G)$. Moreover, the map $$\begin{array}{rcl} R_{\operatorname{par}}(G) & \to & A_{\operatorname{par}}(G)\rtimes_\theta G \\ {[}{g}{]} & \mapsto & \varepsilon_g\delta_g \end{array}$$
    is an isomorphism.
\end{proposition}

\begin{proposition}\label{AparG.universal}
    Let $G$ be a group. Then the algebra $A_{\operatorname{par}}(G)$ is isomorphic to the universal unital $R$-algebra generated by the set $\{\varepsilon_g : g\in G\}$, subject to the relations that the $\varepsilon_g$'s are commuting idempotents and $\varepsilon_e = 1$.
\end{proposition}

Our goal is now to generalize the two results above to the context of group algebras with relations. We begin by extending Proposition~\ref{uni.prop.apar} to \( R_{\operatorname{par}}(G, \mathcal{R}) \).
 
\begin{proposition}\label{quefome}
Let $G$ be a group. For each $g\in G$, let $D_g = A_{\operatorname{par}}(G,\mathcal{R})\varepsilon_g$, and define $\theta_g:D_{g^{-1}}\to D_g$  by $\theta_g(\varepsilon_h\varepsilon_{g^{-1}}) = \varepsilon_{gh}\varepsilon_g$. Then $\theta$ defines a partial action of $G$ on $A_{\operatorname{par}}(G,\mathcal{R})$. Moreover, the map
    $$\begin{array}{rcl} R_{\operatorname{par}}(G,\mathcal{R}) & \to & A_{\operatorname{par}}(G,\mathcal{R})\rtimes_\theta G \\ {[}{g}{]} & \mapsto & \varepsilon_g\delta_g \end{array}$$
    is an isomorphism.
\end{proposition}

\begin{proof}

The steps to prove that $\theta$ is a partial action are identical to those in the case without relations, so we omit them here. The map $\pi: G \to A_{\operatorname{par}}(G,\mathcal{R}) \rtimes_\theta G$ defined by $\pi(g) = 1_g \delta_g$ is a partial representation of $G$ (following the same proof as in the case without relations). Since $\pi(g)\pi(g^{-1}) = (\varepsilon_g\delta_g)(\varepsilon_{g^{-1}}\delta_{g^{-1}}) = \varepsilon_g\delta_e$ and $A_{\operatorname{par}}(G,\mathcal{R})\delta_e$ is isomorphic to $A_{\operatorname{par}}(G,\mathcal{R})$, it follows that $\pi$ satisfies the relations in $\mathcal{R}$. Therefore, there exists a homomorphism $\eta: R_{\operatorname{par}}(G,\mathcal{R}) \to A_{\operatorname{par}}(G,\mathcal{R})\rtimes_\theta G$ such that $[g] \mapsto 1_g\delta_g$.

On the other hand, consider the partial representation $u: G \to R_{\operatorname{par}}(G,\mathcal{R})$ defined by $u(g) = [g]$ and the inclusion $\iota: A_{\operatorname{par}}(G,\mathcal{R}) \to R_{\operatorname{par}}(G,\mathcal{R})$. Since $(\iota, u)$ is $\theta$-covariant (again, by the same reasoning as in the case without relations), there exists a homomorphism  $\iota \times u: A_{\operatorname{par}}(G,\mathcal{R}) \rtimes_\theta G \to R_{\operatorname{par}}(G,\mathcal{R})$ such that $(\iota \times u)(a\delta_g) = \iota(a)u(g) = a[g]$. It is clear that $\eta$ and $\iota \times u$ are inverses of each other.
\end{proof}

Let $\rho:A_{\operatorname{par}}(G) \to A_{\operatorname{par}}(G, \mathcal{R})$ be the surjective homomorphism obtained from the restriction (and co-restricion) of the natural homomorphism of \(R_{\operatorname{par}}(G)\) onto \(R_{\operatorname{par}}(G, \mathcal{R}) \) that sends $[g]$ to $[g]$. It is clear that $\rho$ sends $\varepsilon_g$ to $\varepsilon_g$. To extend Proposition~\ref{AparG.universal} to \( A_{\operatorname{par}}(G, \mathcal{R}) \), we first need to prove a result concerning $\rho$.

\begin{proposition}\label{pressaoalta}
    The kernel of $\rho$, $\ker\rho$, is the smallest $\theta$-invariant ideal of $A_{\operatorname{par}}(G)$ that contains $\mathcal{R}$.
\end{proposition}
\begin{proof}
   If $\sum_i\lambda_i\prod_j\varepsilon_{g_{ij}} = 0$ is a relation in $\mathcal{R}$, then $\sum_i\lambda_i\prod_j\varepsilon_{g_{ij}} \in \ker\rho$, which implies that $\ker\rho$ contains $\mathcal{R}$.
To see that $\ker\rho$ is $\theta$-invariant, let $a = \varepsilon_{g_1}\varepsilon_{g_2}\cdots\varepsilon_{g_n}\varepsilon_{g^{-1}} \in \ker\rho \cap D_{g^{-1}}$. Since $a = 0$ in $A_{\operatorname{par}}(G,\mathcal{R})$, it follows that $a\delta_e = 0$ in $A_{\operatorname{par}}(G,\mathcal{R}) \rtimes_\theta G$. Then
$$
0 = (\varepsilon_{g} \delta_{g})(a \delta_{e})(\varepsilon_{g^{-1}} \delta_{g^{-1}}) = \theta_g(a)\delta_{e},
$$
which implies that $\theta_g(a) = 0$ in $A_{\operatorname{par}}(G,\mathcal{R})$, and hence $\theta_g(a) \in \ker\rho$. By linearity, $\ker\rho$ is $\theta$-invariant.

Now, let $I$ be another ideal of $A_{\operatorname{par}}(G)$ that is $\theta$-invariant and contains $\mathcal{R}$. Since $I$ is $\theta$-invariant, the partial skew group ring $A_{\operatorname{par}}(G)/I \rtimes_\theta G$ is well defined. The map $G \to A_{\operatorname{par}}(G)/I \rtimes_\theta G$ given by $g \mapsto \bar{\varepsilon}_g\delta_g$ defines a partial representation of $G$ that satisfies the relations in $\mathcal{R}$. Consequently, there exists a homomorphism $\tau: R_{\operatorname{par}}(G,\mathcal{R}) \to A_{\operatorname{par}}(G)/I \rtimes_\theta G$ such that $\tau([g]) = \bar{\varepsilon}_g\delta_g$. Restricting $\tau$ to $A_{\operatorname{par}}(G,\mathcal{R})$, we obtain the following chain of homomorphisms:
    $$\begin{array}{ccccccc}
        A_{\operatorname{par}}(G) & \xrightarrow{\rho} & A_{\operatorname{par}}(G,\mathcal{R}) & \xrightarrow{\tau} & (A_{\operatorname{par}}(G)/I)\delta_e & \cong & A_{\operatorname{par}}(G)/I \\
        \varepsilon_g & \mapsto & \varepsilon_g & \mapsto & \bar{\varepsilon}_g\delta_e & \mapsto & \bar{\varepsilon}_g.
      \end{array}$$
    Therefore, $\ker\rho \subseteq I$.
\end{proof}

One can check that the smallest $\theta$-invariant ideal of $A_{\operatorname{par}}(G)$ that contains $\mathcal{R}$ is the ideal generated by the set $$\mathcal{R}_{\operatorname{inv}} = \Big\{\varepsilon_g\sum_i\lambda_i\prod_j\varepsilon_{gg_{ij}} : \sum_i\lambda_i\prod_j\varepsilon_{g_{ij}} \in \mathcal{R}, \ g\in G\Big\}.$$
Then, by the proposition above, $\ker\rho$ is the ideal generated by $\mathcal{R}_{\operatorname{inv}}$.

\begin{corollary}
    The algebra $A_{\operatorname{par}}(G,\mathcal{R})$ is isomorphic to the universal unital $R$-algebra generated by the set $\{\varepsilon_g : g \in G\}$, subject to the relations that the $\varepsilon_g$ are commuting idempotents, $\varepsilon_e = 1$, and the relations in $\mathcal{R}_{\operatorname{inv}}$.
\end{corollary}
\begin{proof}
    This follows from Proposition~\ref{AparG.universal} and the fact that $\ker\rho$ is the ideal generated by $\mathcal{R}_{\operatorname{inv}}$.
\end{proof}

\section{Commutative algebras generated by idempotents}

In the previous section, we identified two commutative algebras generated by idempotents. Keimel describes commutative algebras over commutative unital rings in \cite{Keimel}, with results based on Stone duality. In this section, we briefly recall Stone duality and establish some results that will be necessary to describe the partial group algebra with relations using a topological partial action.

We consider a Boolean algebra as a relatively complemented distributive lattice with a least element. We do not assume the existence of a greatest element. We thus use a generalized version of Stone duality for Boolean algebras, and we view a Stone space as a Hausdorff space with a basis of compact-open sets.

Let $B$ be a Boolean algebra. The dual of $B$, denoted by $\widehat{B}$, is the Stone space of all ultrafilters in $B$, with the topology given by the basis of compact-open sets $\{O_x\}_{x\in B}$, where $O_x := \{\xi \in \widehat{B} : x \in \xi\}$. Conversely, given a Stone space $X$, the family of all compact-open subsets of $X$ forms a Boolean algebra. For more details, see \cite[Section~2.3]{BCGWSubshift} or \cite[Section~2.4]{LawsonStoneDuality}. We note that the Boolean algebra $B$ has a greatest element if and only if $\widehat{B}$ is compact.

Fix a nonzero unital commutative ring $R$. Given a Stone space $X$, we denote by $\Lc(X,R)$ the algebra of locally constant functions with compact support from $X$ to $R$. For an algebra $A$, we let $E(A)$ denote the set of idempotents of $A$.

We begin by considering the universal unital $R$-algebra generated by commuting idempotents. It is not necessarily true that this universal algebra is torsion-free, so we do not obtain that the homomorphisms given by \cite[Theorem~1]{Keimel} is an isomorphism. Nevertheless, under the hypothesis that the ring $R$ is indecomposable, the universal property may be used to find an inverse to the map given in \cite[Theorem~1]{Keimel}.

\begin{proposition}\label{almoco}
    Suppose $R$ is indecomposable, and let $A_S$ be the universal $R$-algebra generated by a non-empty set $S$ with a distinguished element $1$, subject to the relations $x^2=x$ and $xy=yx$ for all $x,y\in S$ and $1$ is a unit. Then $A_S$ is not the zero algebra and the unital $R$-algebra homomorphism $\Psi:A_S\to\Lc(\widehat{E(A_S)},R)$ given by $\Psi(x)=1_{O_x}$ for all $x\in S$ is an isomorphism.
\end{proposition}

\begin{proof}
To prove that $A_S$ is nonzero, consider the algebra $R^S$. For each $x\in S$ with $x\neq 1$, consider $e_x\in R^S$ the element that has $1$ in coordinate $x$ and $0$ otherwise, and consider $e_1$ the unit of $R^S$. Then the family $\{e_x\}_{x\in S}$ satisfies the relations defining $A_S$. Since we are assuming that $R$ is a nonzero ring, it follows that by the universality of $A_S$ that $A_S$ is nonzero.

By \cite[Theorem~1]{Keimel}, there exists a surjective homomorphism $\Upsilon:\Lc(\widehat{E(A_S)},R)\to A_S$ such that $\Upsilon(1_{O_e})=e$ for all $e\in E(A_S)$. By the universal property of $A_S$, there exists a homomorphism $\Psi:A_S\to \Lc(\widehat{E(A_S)},R)$ such that $\Upsilon(x)=1_{O_x}$ for all $x\in S$, since each $1_{O_x}$ is idempotent and $\Lc(\widehat{E(A_S)},R)$ is commutative. Note that $\Upsilon(\Psi(x))=x$ for all $x\in S$. Once again, applying the universal property of $A_S$, we conclude that $\Upsilon\circ\Psi=Id_{A_S}$. In particular, $\Psi$ is injective.

We claim that for every $e\in E(A_S)$, we have that $\Psi(e)=1_{O_e}$. To prove this claim, let $M_S$ be the free abelian monoid generated by $S$ with $1$ being the unit of $M_S$, and note that $A_S$ is generated by $M_S$ as an $R$-module. First, suppose that $e\in M_S$, that is $e=x_1\cdots x_n$ for some $n\in\nn$ and $x_1,\cdots,x_n\in S$. In this case, in the Boolean algebra $E(A_S)$, we have $e=x_1\wedge\cdots\wedge x_n$, and hence
    \[\Psi(e)=1_{O_{x_1}}\cdot\cdots\cdot 1_{O_{x_n}}=1_{O_{x_1\wedge\cdots\wedge x_n}}=1_{O_e}.\]
    
    Now suppose that $e$ is the complement of an element of $M_S$ in $E(A_S)$, that is, $e=1-x_1\cdots x_n$ for some $n\in\nn$, and $x_1,\cdots,x_n\in S$. Applying the previous case, we have that
    \[\Psi(e)=1-1_{O_{x_1\cdots x_n}}=1_{O_{x_1\cdots x_n}^c}=1_{O_{1-x_1\cdots x_n}}=1_{O_e}.\]

   Given that the product in $E(A_S)$ corresponds to the meet as a Boolean algebra as above, it follows that if $e$ is a product of elements of $M_S$ and their complements, then $\Psi(e)=1_{O_e}$.
    
    For the general case, given $e\in E(A_S)\setminus\{0\}$, a usual disjointification procedure (see for instance \cite[Lemma~4.1]{BoavaCastro}) shows that $e=\sum_{i=1}^n r_ip_i$, where $r_i\in R\setminus\{0\}$, $p_i$ is a non zero product of elements of $M_S$ and their complements in $E(A_S)$, and $p_ip_j=0$ for all $i,j=1,\ldots,n$ with $i\neq j$. Given $i\in\{1,\ldots,n\}$, note that $r_i^2p_i=e^2p_i=ep_i=r_ip_i$. Applying the previous case and using that $\Psi$ is injective, we conclude that $r_i^21_{O_{p_i}}=r_i1_{O_{p_i}}$. Since $p_i\neq 0$, we have that $O_{p_i}\neq\emptyset$. Applying the above equality to a point of $O_{p_i}$, we see that $r_i^2=r_i$. Because $R$ is indecomposable and $r_i\neq 0$, we have that $r_i=1$. It follows that
    \[e=p_1+\cdots+p_n=p_1\vee\cdots\vee p_n,\]
    where the last equality is due to the hypothesis that $p_ip_j=0$ if $i\neq j$. Note that $O_{p_i}\cap O_{p_j}=\emptyset$ if $i\neq j$, and hence
    \[\Upsilon(e)=1_{O_{p_1}}+\cdots+1_{O_{p_n}}=1_{O_{p_1}\cup\cdots\cup O_{p_n}}=1_{O_{p_1\vee\cdots\vee p_n}}=1_{O_e}.\]

    Finally, we prove that $\Psi\circ \Upsilon=Id_{\Lc(\widehat{E(A_S)},R)}$. An element $f\in \Lc(\widehat{E(A_S)},R)$ can be written as $f=\sum_{i=1}^nr_i1_{O_{e_i}}$, for some $r_i\in R$ and $e_i\in E(A_S)$ for each $i=1,\ldots,n$. Hence
    \[\Psi(\Upsilon(f))=\Psi\left(\sum_{i=1}^nr_ie_i\right)=\sum_{i=1}^nr_i1_{O_{e_i}}.\]

    It follows that $\Upsilon=\Psi^{-1}$ and $\Psi$ is an isomorphism.
\end{proof}

We will apply Proposition~\ref{almoco} in the context of the previous section. To do so, we consider quotients of \(\Lc(X, R)\) for a given Stone space \(X\).

\begin{lemma}\label{lemma.Phi.surjective}
    Let $X$ be a Stone space and $F\scj X$ a closed set. Then $F$ is also a Stone space and the map $\Phi:\Lc(X,R)\to\Lc(F,R)$ given by $\Phi(f)=f|F$, for $f\in \Lc(X,R)$, is a surjective $R$-algebra homomorphism.
\end{lemma}

\begin{proof}
    Note that $F$ is Hausdorff and the family $\{U\cap F:U\text{ is compact-open in }X\}$ is a basis of compact-open sets for $F$. Hence, $F$ is a Stone space.

    It is clear that $\Phi$ is an $R$-algebra homomorphism. It remains to show that $\Phi$ is surjective. Let $g\in \Lc(F,R)$ be non zero. We can write $g=\sum_{i=1}^n r_i1_{V_i}$, where $r_1,\ldots,r_n$ are non zero pairwise distinct elements of $R$ and $V_i=g^{-1}(r_i)$. Note that each $V_i$ is compact and open relative to $F$. For each $i=1,\ldots,n$, there exists a set $U_i$ open relative to $X$ such that $V_i=U_i\cap F$. Because $U_i$ can be written as a union of compact-open sets of $X$ and $V_i$ is compact, we may assume without loss of generality that $U_i$ is compact-open. For each $i=1,\ldots,n$, consider the set $U'_i=U_i\setminus\bigcup_{j\in\{1,\ldots,n\}\setminus\{i\}}U_j$, which is again compact-open in $X$. Note that $V_i\scj U'_i$, otherwise there would be an $x\in V_i\scj F$ and $j\neq i$ such that $x\in U_j$, which would imply that $x\in V_j$, a contradiction. So, from the start, we may assume that the family $\{U_i\}_{i=1}^n$ is pairwise disjoint. Defining $f=\sum_{i=1}^n r_i1_{U_i}$, we  obtain that $f\in\Lc(X,R)$ and $\Phi(f)=g$.
\end{proof}

\begin{lemma}\label{ideal.not.field}
  Let $X$ be a Stone space. Given an ideal $J$ of $\Lc(X,R)$, we let $F$ be the closed subset of $X$ given by 
\begin{displaymath}
	F=\{x \in X \, \mid \, f(x)=0, \ \forall f \in J \}.
\end{displaymath}
    Also, let $I(F)$ be the ideal given by
    \begin{displaymath}
	I(F)= \{ f \in \Lc(X,R) \, \mid \, f(x) = 0, \, \forall x \in F \}.
\end{displaymath}
Then, $I(F)=J$ if and only if for every $x\in X\setminus F$, there is $f\in J$ such that $f(x)=1$.
\end{lemma}

\begin{proof}
    Suppose first that $I(F)=J$. Since $X\setminus F$ is open, given $x\in X\setminus F$, there is a compact-open neighborhood $V$ of $x$ such that $V\scj X\setminus F$. Then $1_V\in I(F)=J$ and $1_V(x)=1$.

    For the converse, we adapt the proof of \cite[Lemma~4.2]{BeuterIdeal}. Clearly, $J\scj I(F)$. Let $f\in I(F)$ and note that we can write $f=\sum_{i=1}^nr_i1_{V_i}$, where $r_i\in R$ and $V_i\scj X\setminus F$ is compact-open for $i=1,\ldots,n$. By hypothesis, for each $i=1,\ldots,n$ and $x\in V_i$, there exists $f_{i,x}\in J$ such that $f_{i,x}(x)=1$. Since $f_{i,x}$ is locally constant and $J$ is an ideal, by multiplying by a suitable characteristic function, we may assume that $f_{i,x}=1_{U_{i,x}}$ for some compact-open set $U_{i,x}\scj V_i$. Note that for any compact-open set $U$ contained in some $U_{i,x}$, we again have that $1_U\in J$. Using that $V_i$ is compact and that compact-open sets form a Boolean algebra, we can write $V_i$ as a disjoint union of compact-open sets $U_{i,1},\cdots,U_{i,n_i}$ such that $1_{U_{i,j}}\in J$ for all $j=1,\ldots,n_i$. Hence $1_{V_i}=\sum_{j=1}^{n_i}1_{U_{i,j}}\in J$, from where it follows that $f\in J$.
\end{proof}

The two lemmas above allow us to identify quotients of \(\Lc(X, R)\) for a given Stone space \(X\).

\begin{proposition}\label{quotient.not.field}
    Let $X$ be a Stone space and $J$ an ideal of $\Lc(X,R)$. Let $F$ be as in Lemma~\ref{ideal.not.field} and suppose that for every $x\in X\setminus F$, there is $f\in J$ with $f(x)=1$. Then $\Lc(X,R)/J\cong \Lc(F,R)$.
\end{proposition}

\begin{proof}
    By Lemma~\ref{lemma.Phi.surjective}, the map $\Phi:\Lc(X,R)\to \Lc(F,R)$ given by $\Phi(f)=f|_F$ is a surjective $R$-algebra homomorphism. Note that $\ker(\Phi)=I(F)$, where $I(F)$ is defined in Lemma~\ref{ideal.not.field}. The result now follows from Lemma~\ref{ideal.not.field} and the first isomorphism theorem.
\end{proof}

\begin{remark}\label{starving}
If \(R\) is a field, then Lemma~\ref{ideal.not.field} and Proposition~\ref{quotient.not.field} hold without the need of adding the condition on \(J\) that for every \(x \in X \setminus F\), there exists \(f \in J\) such that \(f(x) = 1\). Indeed, for every \(x \in X \setminus F\), by the definition of \(F\), there exists \(f \in J\) such that \(f(x) \neq 0\). Multiplying \(f\) by \(f(x)^{-1}\), we obtain a new function \(g \in J\) such that \(g(x) = 1\).
\end{remark}

We conclude this section by proving that when the base ring is indecomposable, the hypothesis of Proposition~\ref{quotient.not.field} is satisfied if $J$ is the kernel of a homomorphism between the algebras of locally constant functions of two Stone spaces. This will be the case in Section~\ref{s:subshift}.

\begin{lemma}\label{bottle}
   Suppose \( R \) is an indecomposable ring. Let \( X \) and \( Y \) be compact Stone spaces, and \( \phi: \Lc(X, R) \to \Lc(Y, R) \) a unital \( R \)-algebra homomorphism. Define 
\[
F := \{ x \in X : f(x) = 0 \text{ for all } f \in \ker(\phi) \}.
\]
Then, for every \( x \in X \setminus F \), there exists \( f \in \ker(\phi) \) such that \( f(x) = 1 \). Moreover, if \( \phi \) is surjective, then \( F \), with the subspace topology, is homeomorphic to \( Y \).

\end{lemma}

\begin{proof}

Since \( R \) is indecomposable, we can adapt the proof of \cite[Proposition~2.13]{BCGWSubshift} to find a continuous function \( h: Y \to X \) such that \(\phi(f) = f \circ h\) for all \( f \in \Lc(X, R) \). Given that \( Y \) is compact, the image \( h(Y) \) is closed in \( X \). We claim that \( F = h(Y) \).

First, if \( x = h(y) \) for some \( y \in Y \), then for every \( f \in \ker(\phi) \), we have \( f(x) = f(h(y)) = \phi(f)(y) = 0 \). Conversely, if \( x \notin h(Y) \), then since \( h(Y) \) is closed, there exists a compact-open neighborhood \( U \) of \( x \) such that \( U \subseteq X \setminus h(Y) \). In this case, \( 1_U \in \ker(\phi) \) and \( 1_U(x) = 1 \neq 0 \), so \( x \notin F \). Hence, \( F = h(Y) \), and the first part of the statement follows.

Now, suppose \( \phi \) is surjective, and let us show that \( h \) is injective. Take \( y_1, y_2 \in Y \) with \( y_1 \neq y_2 \). Since \( Y \) is a Stone space, there exists a compact-open subset \( V \) of \( Y \) such that \( y_1 \in V \) and \( y_2 \notin V \). Because \( \phi \) is surjective, there exists \( f \in \Lc(X, R) \) such that \( \phi(f) = 1_V \). Thus, \( f(h(y_1)) = 1_V(y_1) = 1 \) and \( f(h(y_2)) = 1_V(y_2) = 0 \), implying \( h(y_1) \neq h(y_2) \). Therefore, the co-restriction \( h|^F: Y \to F \) is bijective. Since \( Y \) is compact and \( F \) is Hausdorff, \( h|^F \) is a homeomorphism.
\end{proof}

\section{Topological characterization of Partial Group Algebras}

As in Section~\ref{greve}, let $G$ be a group, let $\mathcal{R}$ be a set of relations on the set $\{\varepsilon_g:g\in G\}$, let $R$ a nonzero unital commutative ring, and consider the partial group algebras $R_{\operatorname{par}}(G)$ and $R_{\operatorname{par}}(G, \mathcal{R})$, as well as their $R$-subalgebras $A_{\operatorname{par}}(G)$ and $A_{\operatorname{par}}(G,\mathcal{R})$. The main goal of this section is to describe $A_{\operatorname{par}}(G,\mathcal{R})$ as the set of locally constant functions on some topological space and use it to obtain a topological partial action whose associated partial skew group ring is isomorphic to $R_{\operatorname{par}}(G, \mathcal{R})$. In particular, the same characterization holds for $A_{\operatorname{par}}(G)$ and $R_{\operatorname{par}}(G)$.

Recall that there is a canonical bijection between the power set of $G$ and $\{0,1\}^G$. Give $\{0,1\}$ the discrete topology. Then we topologize the power set of $G$ by pulling back the product topology on $\{0,1\}^G$ via the aforementioned bijection.  

Let $X_G = \{\xi\subseteq G : e\in\xi\}$ and $X_g = \{\xi\in X_G : g\in\xi\}$ for each $g\in G$. Observe that $X_G$ and $X_g$ are compact with the subspace topology inherited from the power set of $G$. For each $g\in G$, denote by $1_g\in \Lc(X_G,R)$ the characteristic function of $X_g$.

\begin{proposition}\label{iso.theta}
    Assume that $R$ is an indecomposable ring. Then, there is an isomorphism $\Theta:A_{\operatorname{par}}(G) \to \Lc(X_G,R)$ such that $\Theta(\varepsilon_g)=1_g$ for all $g\in G$.
\end{proposition}
\begin{proof}

By the universality of $A_{\operatorname{par}}(G)$, as given by Proposition~\ref{AparG.universal}, there exists an $R$-algebra homomorphism $\Theta: A_{\operatorname{par}}(G) \to \Lc(X_G, R)$ such that $\Theta(\varepsilon_g) = 1_g$ for each $g \in G$.

To construct the inverse, note that there is a continuous function $h: \reallywidehat{E(A_{\operatorname{par}}(G))} \to X_G$ given by
\[
h(\mathcal{F}) = \{g \in G : \varepsilon_g \in \mathcal{F}\}.
\]
This induces a homomorphism $\Lambda: \Lc(X_G, R) \to \Lc(\reallywidehat{E(A_{\operatorname{par}}(G))}, R)$ defined by $\Lambda(f) = f \circ h$. For $\mathcal{F} \in \reallywidehat{E(A_{\operatorname{par}}(G))}$, we have
\[
\Lambda(1_g)(\mathcal{F}) = 1 \Leftrightarrow g \in h(\mathcal{F}) \Leftrightarrow \varepsilon_g \in \mathcal{F} \Leftrightarrow \mathcal{F} \in O_{\varepsilon_g}.
\]
Thus, $\Lambda(1_g) = 1_{O_{\varepsilon_g}}$.

Using the universality of $A_{\operatorname{par}}(G)$ given in Proposition~\ref{AparG.universal}, and Proposition~\ref{almoco}, we obtain an isomorphism $\Psi:A_{\operatorname{par}}(G)\to \Lc(\reallywidehat{E(A_{\operatorname{par}}(G))},R)$ such that $\Psi(\varepsilon_g) = 1_{O_{\varepsilon_g}}$. Therefore, the composition $\Psi^{-1}\circ \Lambda:\Lc(X_G,R)\to A_{\operatorname{par}}(G)$ satisfies $\Psi^{-1}\circ \Lambda(1_g)=\varepsilon_g$ for all $g\in G$.

Since $\Lc(X_G, R)$ is generated by the functions $1_g$ for $g \in G$, we conclude that $\Psi^{-1} \circ \Lambda = \Theta^{-1}$.
\end{proof}

Assuming that $R$ is indecomposable, and applying the  isomorphism $\Theta:A_{\operatorname{par}}(G) \to \Lc(X_G,R)$ given in Proposition~\ref{iso.theta} to each relation in $\mathcal{R}$, we obtain the set
$$\Theta(\mathcal{R}) = \Big\{\sum_i\lambda_i\prod_j1_{g_{ij}} : \sum_i\lambda_i\prod_j\varepsilon_{g_{ij}} \in \mathcal{R}\Big\}.$$
Define
$$\Omega_{\mathcal{R}} = \{\xi\in X_G : f(g^{-1}\xi)=0, \forall\,f\in\Theta(\mathcal{R}) \mbox{ and } g\in \xi\},$$
which is a closed set of $X_G$. Consider the surjective homomorphism $\Phi:\Lc(X_G,R)\to\Lc(\Omega_{\mathcal{R}},R)$ given by Lemma~\ref{lemma.Phi.surjective}. Using the notation of Lemma~\ref{ideal.not.field}, $\ker \Phi = I(\Omega_{\mathcal{R}})$.

Let $\rho: A_{\operatorname{par}}(G) \to A_{\operatorname{par}}(G, \mathcal{R})$ be the canonical homomorphism from Section~\ref{greve}. Recall that its kernel, $\ker \rho$, is the smallest $\theta$-invariant ideal of $A_{\operatorname{par}}(G)$ that contains $\mathcal{R}$ (Proposition~\ref{pressaoalta}) and is the ideal generated by $\mathcal{R}_{\operatorname{inv}}$. Using the isomorphism $\Theta$ above, we may consider the ideal $\Theta(\ker \rho)$ in $\Lc(X_G, R)$.

\begin{proposition}\label{knot}
    Suppose that $R$ is an indecomposable ring. Then, $\{\xi\in X_G \, \mid \, f(\xi)=0, \ \forall f \in \Theta(\ker \rho)\} = \Omega_{\mathcal{R}}$.
\end{proposition}
\begin{proof}
   
    Denote by $F$ the set on the left-hand side of the set equality in the proposition. In what follows, we use the notation $[S]$, which equals 1 if the statement $S$ is true and 0 otherwise.

 Let $\xi\in F$, $g\in\xi$ and $f=\sum_i\lambda_i\prod_j1_{g_{ij}}\in\Theta(\mathcal{R})$. Then,
\begin{align*}
    f(g^{-1}\xi) & = \sum_i\lambda_i\prod_j1_{g_{ij}}(g^{-1}\xi) \\
    & = \sum_i\lambda_i\prod_j[g_{ij}\in g^{-1}\xi] \\
    & = \sum_i\lambda_i\prod_j[gg_{ij}\in\xi] \\
    & =\sum_i\lambda_i\prod_j1_{gg_{ij}}(\xi) \\
     & = 1_g\sum_i\lambda_i\prod_j1_{gg_{ij}}(\xi) \\
    & = 0,
\end{align*}
since $g\in\xi$ and $1_g\sum_i\lambda_i\prod_j1_{gg_{ij}}\in \Theta(\ker \rho)$. This shows that $\xi\in \Omega_{\mathcal{R}}$.

Now, let $\xi\in\Omega_{\mathcal{R}}$ and $f \in \Theta(\ker \rho)$. Since $\ker \rho$ is generated by $\mathcal{R}_{\operatorname{inv}}$, we can write $f = 1_g\sum_i\lambda_i\prod_j1_{gg_{ij}}$.
    Note that 
    \begin{align*}
    f(\xi) & = 1_g\sum_i\lambda_i\prod_j1_{gg_{ij}}(\xi) \\
    & = [g\in\xi]\sum_i\lambda_i\prod_j[gg_{ij}\in\xi] \\
    & = [g\in\xi]\sum_i\lambda_i\prod_j[g_{ij}\in g^{-1}\xi] \\
    & = [g\in\xi]\sum_i\lambda_i\prod_j1_{g_{ij}}(g^{-1}\xi) \\
    & = 0,
\end{align*}
 where the last equality is clearly true if $g\notin\xi$. If $g\in\xi$, then $\sum_i\lambda_i\prod_j1_{g_{ij}}(g^{-1}\xi) = 0$ because $\sum_i\lambda_i\prod_j1_{g_{ij}}\in\Theta(\mathcal{R})$ and $\xi\in\Omega_{\mathcal{R}}$.
\end{proof}

Under the hypothesis that \(R\) is indecomposable, we have shown that 
\[
\Theta(\ker \rho) \subseteq I(\Omega_{\mathcal{R}}) = \ker \Phi.
\]
By Lemma~\ref{ideal.not.field}, this inclusion is an equality if and only if, for every \(\xi \in X_G \setminus \Omega_{\mathcal{R}}\), there exists an \(f \in \Theta(\ker \rho)\) such that \(f(\xi) = 1\). Assuming the latter condition (which always holds if $R$ is a field as observed in Remark~\ref{starving}), and applying Proposition~\ref{quotient.not.field}, we obtain an isomorphism 
\[
\Theta_{\mathcal{R}}: A_{\operatorname{par}}(G, \mathcal{R}) \to \Lc(\Omega_{\mathcal{R}}, R)
\]
such that \(\Theta_{\mathcal{R}}(\varepsilon_g) = 1_g\) for all \(g \in G\).  However, $A_{\operatorname{par}}(G, \mathcal{R})$ is not always isomorphic to $\Lc(\Omega_{\mathcal{R}}, R)$, as illustrated in the following example.

\begin{example}
    Suppose that $R$ is not a field and let $r\in R$ be nonzero and non-invertible. Consider the set of relations $\mathcal{R}=\{r\varepsilon_e\}$. By Proposition~\ref{pressaoalta}, and the comment after its proof, $\ker\rho$ is the ideal generated by $\{r\varepsilon_g:g\in G\}$. Using the isomorphism of Proposition~\ref{iso.theta}, we see that elements of $\ker\rho$ are of the form $rf$ for some $f\in\Lc(X_G,R)$. Since $r$ is not invertible, we have that $1\notin\ker\rho$ so that $\ker\rho$ is not the whole algebra. This implies that $A_{\operatorname{par}}(G,\mathcal{R})$ is not the zero algebra and that $r1=0$ in $A_{\operatorname{par}}(G,\mathcal{R})$. Hence $A_{\operatorname{par}}(G,\mathcal{R})$ cannot be isomorphic to $\Lc(X,R)$ for any Stone space $X$, as $r1_X\neq 0$ unless $X=\emptyset$.
\end{example}

Using the isomorphism $\Theta_{\mathcal{R}}$, we can translate the partial action \(\theta\) from Proposition~\ref{quefome} to the algebra \(\Lc(\Omega_{\mathcal{R}}, R)\). Furthermore, this partial action $\theta$ can be obtained from a topological partial action $\widehat{\theta}$ on $\Omega_{\mathcal{R}}$. We summarize this and the main results obtained so far in the following theorem.

\begin{theorem}\label{grevedosalunos}
    Let $R$ be a nonzero, unital,  commutative, indecomposable ring, $G$ be a group, and $\mathcal{R}$ a set of relations on the set $\{\varepsilon_g:g\in G\}$. Suppose that for every $\xi\in X_G\setminus \Omega_{\mathcal{R}}$, there is $f\in \Theta(\ker \rho)$ such that $f(\xi)=1$ (which holds when $R$ is a field, see Remark~\ref{starving}). Then:
    \begin{enumerate}[(i)]
        \item There is an isomorphism $\Theta_{\mathcal{R}}:A_{\operatorname{par}}(G,\mathcal{R}) \to \Lc(\Omega_{\mathcal{R}},R)$ such that $\Theta_{\mathcal{R}}(\varepsilon_g) = 1_g$ for all $g\in G$.
        \item  For each $g\in G$, let $\Omega_g = \{\xi\in \Omega_{\mathcal{R}} \mid g\in\xi\}$ and $\widehat{\theta}_g:\Omega_{g^{-1}}\to\Omega_{g}$ be given by $\widehat{\theta}_g(\xi)=g\xi$. Then $\widehat{\theta}$ is a topological partial action on $\Omega_{\mathcal{R}}$.
        \item For each $g\in G$, let $D_g = \{f\in \Lc(\Omega_{\mathcal{R}},R) \mid f(\xi) = 0, \, \forall \xi\notin\Omega_g\}$, denote by $1_g\in\Lc(\Omega_{\mathcal{R}},R)$ the characteristic function of $\Omega_g$, and define $\theta_g : D_{g^{-1}}\to D_g$ by $\theta_g(f) =1_{g^{-1}} (f\circ\widehat{\theta}_{g^{-1}})$. Then, $\theta$ is a partial action of $G$ on $\Lc(\Omega_{\mathcal{R}},R)$ such that $\theta_g(1_h1_{g^{-1}}) = 1_{gh}1_g$ for all $g,h\in G$. Moreover, $\theta$ is equivalent to the partial action of Theorem~\ref{quefome} under the isomorphism $\Theta_{\mathcal{R}}$.
        \item The map $$\begin{array}{rcl} R_{\operatorname{par}}(G,\mathcal{R}) & \to & \Lc(\Omega_{\mathcal{R}},R)\rtimes_\theta G \\ {[}{g}{]} & \mapsto & 1_g\delta_g \end{array}$$ is an isomorphism.
    \end{enumerate}

\end{theorem}

We can use Theorem~\ref{grevedosalunos} to prove that $R_{\operatorname{par}}(G,\mathcal{R})$ is isomorphic to the Steinberg algebra of the transformation groupoid $G\ltimes_{\widehat{\theta}}\Omega_{\mathcal{R}}$ (see \cite{MR3743184} for more details). This can be seen as a generalization of \cite[Corollary~2.7]{MEP}, which describes $R_{\operatorname{par}}(G)$ as a groupoid algebra when $G$ is finite.

\begin{corollary}\label{SteinbergAgain}
    Under the hypothesis of Theorem~\ref{grevedosalunos}, the algebra $R_{\operatorname{par}}(G,\mathcal{R})$ is isomorphic to the Steinberg algebra $\mathcal{A}_R(G\ltimes_{\widehat{\theta}}\Omega_{\mathcal{R}})$.
\end{corollary}

\begin{proof}
    It follows from Theorem~\ref{grevedosalunos} and \cite[Theorem~3.2]{MR3743184}.
\end{proof}

\section{Application to Subshift algebras}\label{s:subshift}

In this section, we show that the subshift algebras defined in \cite{BCGWSubshift} can be realized as partial group algebras with relations. For the reader's convenience, we restate the definition of a subshift and its algebra, along with basic properties of the algebra. For further details, see \cite{BCGWSubshift}. 

Fix a nonzero unital commutative ring $R$. Let $\N=\{0,1,2,3,\ldots\}$ and let $\alf$ be a non-empty set, called an \emph{alphabet}. The \emph{shift map} on $\alf^\N$ is the map $\sigma: \alf^\N\to \alf^\N$ given by $\sigma(x)_n=x_{n+1}$. Elements of $\alf^*:=\bigcup_{k=0}^\infty \alf^k$ are called \emph{blocks} or \emph{words}, and $\omega$ denotes the empty word. The length of $\alpha\in\alf^*\cup\alf^{\N}$  is denoted by $|\alpha|$. If $\beta\in\alf^*$, then $\beta\alpha$ denotes the concatenation of $\beta$ and $\alpha$. 

Given $F\subseteq \alf^*$, we define the \emph{subshift} (or \emph{shift space}) $\osf_F\subseteq \alf^\N$ as the set of all sequences $x\in \alf^\N$ such that no finite word from $x$ belongs to $F$. In this paper, we drop the subscript $F$ and write $\osf$ in place of $\osf_F$, as we do not use $F$ explicitly. 

A subset $\osf \subseteq \alf^\N$ is \emph{(shift) invariant} for $\sigma$ if $\sigma (\osf)\subseteq \osf$. For an invariant subset $\osf \subseteq \alf^\N$, we define $\CL_n(\osf)$ as the set of all words of length $n$ that appear in some sequence of $\osf$. The \emph{language} of $\osf$ is set $$\lang:=\bigcup_{n=0}^\infty\CL_n(\osf).$$ If $\osf=\osf_F$ for a set $F$ of forbidden words, then $\osf$ is shift invariant. Given a subshift $\osf$ over an alphabet $\alf$ and $\alpha,\beta\in \lang$, we define \[C(\alpha,\beta):=\{\beta x\in\osf:\alpha x\in\osf\}.\]

\begin{remark}\label{pausa}
    The definition of $C(\alpha,\beta)$ implies that if $\alpha=\alpha_1\ldots\alpha_n$, then
    \[C(\eword,\alpha_1)\cap C(\eword,\alpha_1\alpha_2)\cdots\cap C(\eword,\alpha_1\ldots\alpha_n)=C(\eword,\alpha)\]
    and
    \[C(\alpha_n,\eword)\cap C(\alpha_{n-1}\alpha_n,\eword)\cap\cdots\cap C(\alpha_1\ldots\alpha_n,\eword)=C(\alpha,\eword).\]
\end{remark}

\begin{definition}\label{diachuvoso}
For a subshift $\osf$, we let $\TCB$ be the Boolean algebra of subsets of $\osf$ generated by all $C(\alpha,\beta)$ for $\alpha,\beta\in\lang$. Its Stone dual is denoted by $\HTCB$. For each $A\in \TCB$, we let $O_A:=\{\mathcal{F}\in\HTCB:A\in\mathcal{F}\}$. Then the family $\{O_A\}_{A\in\TCB}$ is a basis of compact open-sets for the topology on $\HTCB$.
\end{definition}

\begin{remark}\label{U_elements}
Notice that each element of $\TCB$ is a finite union of elements of the form \[C(\alpha_1,\beta_1)\cap\ldots\cap C(\alpha_n,\beta_n)  \cap C(\mu_1,\nu_1)^c\cap \ldots \cap C(\mu_m,\nu_m)^c.\]
\end{remark}

\begin{definition}\label{gelado}\cite[Definition 3.2]{BCGWSubshift}
Let $\osf$ be a subshift. We define the \emph{unital subshift algebra} $\ualgshift$ as the universal unital $R$-algebra  with generators $\{p_A: A\in\TCB\}$ and $\{s_a,s_a^*: a\in\alf\}$, subject to the relations:
\begin{enumerate}[(i)]
    \item $p_{\osf}=1$, $p_{A\cap B}=p_Ap_B$, $p_{A\cup B}=p_A+p_B-p_{A\cap B}$ and $p_{\emptyset}=0$, for every $A,B\in\TCB$;
    \item $s_as_a^*s_a=s_a$ and $s_a^*s_as_a^*=s_a^*$ for all $a\in\alf$;
    \item $s_{\beta}s^*_{\alpha}s_{\alpha}s^*_{\beta}=p_{C(\alpha,\beta)}$ for all $\alpha,\beta\in\lang$, where $s_{\eword}:=1$ and, for $\alpha=\alpha_1\ldots\alpha_n\in\lang$, $s_\alpha:=s_{\alpha_1}\cdots s_{\alpha_n}$ and $s_\alpha^*:=s_{\alpha_n}^*\cdots s_{\alpha_1}^*$.
\end{enumerate}
\end{definition}

\begin{proposition}\label{lemma.algebra.unital}\cite[Proposition 3.6]{BCGWSubshift}
Let $\osf$ be a subshift and $\ualgshift$ it associated unital subshift algebra. Then,
\begin{enumerate}[(i)]
    \item $s_a^*s_b= \delta_{a,b} p_{F_a}$, for all $a,b\in \alf$;
    \item $s_\alpha^*s_\alpha$ and $s_\beta^*s_\beta$ commute for all $\alpha,\beta\in\lang$;
    \item $s_\alpha^*s_\alpha$ and $s_\beta s_\beta^*$ commute for all $\alpha,\beta\in\lang$;
    \item $s_\alpha s_\beta=0$ for all $\alpha,\beta\in\lang$ such that $\alpha\beta\notin\lang$;
    \item $\ualgshift$ is generated as an $R$-algebra by the set $\{s_a, s_a^*: a\in\alf\}\cup\{1\}$.
\end{enumerate}
\end{proposition}

Let $\osf$ be a subshift over an alphabet $\alf$. For $A\subseteq \osf$, we let $1_A$ denote the characteristic function on $A$, and we let $\udalgshift= \mathrm{span}_R\{1_{C(\alpha,\beta)}: \alpha,\beta \in\lang\}$. Let $\F$ be the free group generated by $\alf$ with the empty word $\eword$ as the identity. 
\begin{definition}
We say that an element $g$ of $\F$ is {\it simple with respect to $\osf$} if its reduced form is equal $\alpha \beta^{-1}$, with $\alpha,\beta \in \lang$.    
\end{definition}
 
In \cite[Section 5.1]{BCGWSubshift}, a partial action $\tau$ of $\F$ on $\udalgshift$ is contructed such that $\ualgshift$ is isomorphic to the partial skew group ring $\udalgshift\rtimes_{\tau}\F$. The partial action $\tau$ is constructed as follows. 

For $a\in\alf$, we define $\tauh_{a}:C(a,\eword)\to C(\eword,a)$ by $\tauh_{a}(x)=a x$, 
and $\tauh_{a^{-1}}:C(\eword,a)\to C(a,\eword)$ by  $\tauh_{a^{-1}}(ax)= x.$

The maps $\tauh_a$ and $\tauh_{a^{-1}}$ define a unique orthogonal and semi-saturated set-theoretic partial action
$\tauh=\left(\{W_g\}_{g\in\F},\{\tauh_g\}_{g\in\F}  \right)$ of $\F$ on $\osf$ such that if $g$ is simple with respect to $\osf$ and its reduced form is $g=\beta\alpha^{-1}$, then $W_{\beta\alpha^{-1}} = C(\alpha,\beta)$ and
\[\tauh_{\alpha\beta^{-1}}(\beta x)=\alpha x\]
for every $\beta x\in C(\alpha,\beta)$. If $g$ is not simple with respect to $\osf$, then $W_g=\emptyset$.

The partial action $\tauh$ induces a partial action $\tau$ of $\F$ on $\udalgshift$: for $\alpha,\beta\in \lang$ such that $\alpha\beta^{-1}$ is in reduced form, let $D_{\alpha\beta^{-1}}$ the ideal of $\udalgshift$ generated by $1_{C(\beta,\alpha)}$. Define $\tau_{\alpha\beta^{-1}}:D_{\beta\alpha^{-1}}\to D_{\alpha\beta^{-1}}$ by $\tau_{\alpha\beta^{-1}}(f)=f\circ\tauh_{\beta\alpha^{-1}}$, where $f\in D_{\beta\alpha^{-1}}$. Then $\tau_{\alpha\beta^{-1}}$ is an isomorphism that maps the ideal $D_{\beta\alpha^{-1}}$ onto the ideal $D_{\alpha\beta^{-1}}$. If $g$ is not simple with respect to $\osf$, we define $D_g=\{0\}$ and $\tau_g$ equals the zero function. Hence, we have an algebraic partial action $\tau=\left( \{D_g\}_{g\in \F}, \{\tau_g\}_{g\in\F} \right)$  of $\F$ on $\udalgshift$. 

\begin{remark}\label{pizza}
    Proposition 3.19 and Proposition 5.20 of \cite{BCGWSubshift} imply that $\udalgshift$ is isomorphic to the diagonal subalgebra of $\ualgshift$; that is, $\udalgshift \cong \mathrm{span}_R \{p_{C(\alpha, \beta)} : \alpha,\beta\in \lang\}$ via a map that sends $1_{C(\alpha,\beta)}$ to $p_{C(\alpha,\beta)}$. Hence, we identify and work with $\udalgshift$ as the diagonal subalgebra of $\ualgshift$. 
\end{remark}

In \cite[Section 5.2]{BCGWSubshift}, a topological partial action $\varphih=\left(\{V_g\}_{g\in\F},\{\varphih_g\}_{g\in\F}  \right)$ of $\F$ on $\HTCB$ is defined. We omit details of the definition here. For this paper, it suffices to know that  $V_g=\emptyset$ if $g$ is not simple with respect to $\osf$, and $V_{\alpha\beta^{-1}}=1_{O_{C(\beta,\alpha)}}$ for $\alpha,\beta\in\lang$ such that $\alpha\beta^{-1}$ is in reduced form. We also need the following proposition.

\begin{proposition}\cite[Proposition~5.17]{BCGWSubshift}\label{iotamap}
    The function $\iota:X\to \HTCB$ that maps $x$ to $\mathcal{F}_x:=\{A\in\TCB:x\in A\}$ is injective, has a dense image, and is equivariant with respect to the partial actions $\tauh$ and $\varphih$.
\end{proposition}

\begin{proposition}\label{partial.rep}
    The map $\pi:\F\to\ualgshift$ given by
    $$\pi(g) = \left\{\begin{array}{ll} s_\alpha s_\beta^*, & \mbox{if } g \mbox{ is simple with respect to } \osf \mbox{ and its reduced form is } \alpha\beta^{-1}, \\ 0, & {otherwise} \end{array}\right.$$
    is a partial representation.
\end{proposition}

\begin{proof}
We will prove that Conditions~(i) to (iii) of Definition~\ref{partialrep} are satisfied. 
Since $\omega = \omega\omega^{-1}$ and $s_\omega=s_\omega^*=1$, we have that Condition~(i) is satisfied. 

To prove the next two items, observe first that $g\in \F$ is simple with respect to $\osf$ if, and only if, $g^{-1}$ is simple with respect to $\osf$.

To check Condition~(ii), that is,  that $\pi(g)\pi(h)\pi(h^{-1})=\pi(gh)\pi(h^{-1})$, we need to consider several cases. \\
{\it Case 1:} $h$ is not simple with respect to $\osf$. In this case, $\pi(h^{-1})=0$ and Condition~(ii) is satisfied. \\
{\it Case 2:} $h$ is simple with respect to $\osf$ with reduced form $\alpha\beta^{-1}$, but $g$ and $gh$ are not simple with respect to $\osf$. In this case, $\pi(g)$ and $\pi(gh)$ are equal $0$ and Condition~(ii) is satisfied. \\
{\it Case 3:} $h$ and $gh$ are simple with respect to $\osf$, but $g$ is not simple with respect to $\osf$. In this case, since $\pi(g)=0$, we have to show that $\pi(gh)\pi(h^{-1})$ equals $0$. 
Suppose that $\alpha_1\beta_1^{-1}$ is the reduced form of $gh$ and $\alpha_3\beta_3^{-1}$ is the reduced form of $h$. 
Since $g=ghh^{-1}$ and $g$ is not simple with respect to $\osf$, the reduced form of $g$ is $\alpha_1 \beta_2^{-1} \beta_4 \alpha_3^{-1}$ for some nonempty $\beta_2$ and $\beta_4$ such that the first letters of $\beta_2$ and $\beta_4$ are different. Moreover, we have that $\beta_1=\alpha \beta_2$ and $\beta_3=\alpha\beta_4$, for some $\alpha$. Using Proposition~\ref{lemma.algebra.unital}, we obtain
$$\pi(gh)\pi(h^{-1}) = s_{\alpha_1} s_{\beta_1}^* s_{\beta_3}s_{\alpha_3}^* = s_{\alpha_1} s_{\beta_2}^*s_{\alpha}^*s_\alpha s_{\beta_4}s_{\alpha_3}^* = $$ $$s_{\alpha_1} s_{\beta_2}^*s_{\alpha}^*s_\alpha s_{\beta_4}s_{\beta_4}^*s_{\beta_4}s_{\alpha_3}^* = s_{\alpha_1} s_{\beta_2}^*s_{\beta_4}s_{\beta_4}^*s_{\alpha}^*s_\alpha s_{\beta_4}s_{\alpha_3}^* = 0,$$
finishing this case. \\
{\it Case 4:} $h$ and $g$ are simple with respect to $\osf$, but $gh$ are not. In this case, writing $g=\alpha_1\beta_1^{-1}$ and $h=\alpha_2\beta_2^{-1}$ in reduced form, neither $\beta_1$ nor $\alpha_2$ are the beginning of each other. It follows from Proposition~\ref{lemma.algebra.unital}(i) that $\pi(g)\pi(h) = s_{\alpha_1}s_{\beta_1}^*s_{\alpha_2}s_{\beta_2}^* =0$. Since $gh$ cannot be written as $\alpha\beta^{-1}$, we obtain that $\pi(gh)=0$, completing the proof in this case.
{\it Case 5:} $h$, $g$ and $gh$ are simple with respect to $\osf$. In this case, writing $g=\alpha_1\beta_1^{-1}$ and $h=\alpha_2\beta_2^{-1}$ in reduced form, we must have $\alpha_2=\beta_1\gamma$ or $\beta_1=\alpha_2\gamma$. If $\alpha_2=\beta_1\gamma$, then $gh = (\alpha_1\gamma)\beta_2^{-1}$ is the reduced form of $gh$ and, since 
$gh$ is simple with respect to $\osf$, then $\alpha_1\gamma\in \lang$. By Proposition~~\ref{lemma.algebra.unital},
$$\pi(gh)\pi(h^{-1}) = s_{\alpha_1\gamma}s_{\beta_2}^*s_{\beta_2}s_{\beta_1\gamma}^* = s_{\alpha_1}s_{\gamma}s_{\beta_2}^*s_{\beta_2}s_{\gamma}^*s_{\beta_1}^* = 
s_{\alpha_1}s_{\gamma}s_{\beta_2}^*s_{\beta_2}s_{\gamma}^*s_{\beta_1}^* s_{\beta_1} s_{\beta_1}^* $$ $$ =
s_{\alpha_1}s_{\beta_1}^* s_{\beta_1}s_{\gamma}s_{\beta_2}^*s_{\beta_2}s_{\gamma}^* s_{\beta_1}^* = 
s_{\alpha_1}s_{\beta_1}^* s_{\alpha_2}s_{\beta_2}^*s_{\beta_2}s_{\alpha_2}^* = \pi(g)\pi(h)\pi(h^{-1}),$$
as desired. The case $\beta_1=\alpha_2\gamma$ is analogous.

We have exhausted all possible cases and so completed the proof of Condition~(ii) in Definition~\ref{partialrep}. The proof of Condition~(iii) is analogous to the above and we leave it to the reader. 
\end{proof}

\begin{remark}\label{firework}
    Notice that $\pi(\alpha \beta^{-1})\neq 0$ if and only if $C(\alpha,\beta)\neq \emptyset$. Moreover, in this case, $\pi(\beta\alpha^{-1})\pi(\alpha\beta^{-1})=s_{\beta}s^*_{\alpha}s_{\alpha}s^*_{\beta}=p_{C(\alpha,\beta)}$.
\end{remark}

\begin{theorem}\label{clouds}
    Let $\osf$ be a subshift over an alphabet $\alf$, let $\F$ be the free group generated by $\alf$, let be $R$ a commutative unital ring, and let $\pi$ be the partial representation of Proposition~\ref{partial.rep}. Consider the set of relations
    \[\mathcal{R}=\left\{\sum_{i}\lambda_i\prod_j\varepsilon_{\beta_{ij}\gamma_{ij}^{-1}}:\sum_{i}\lambda_i\prod_{j}p_{C(\gamma_{ij},\beta_{ij})}=0\text{ in }\ualgshift\right\}\cup\left\{\varepsilon_g:\pi(g)=0\right\}.
    \]
    Then, there exists an isomorphism $\Xi:R_{\operatorname{par}}(\F,\mathcal{R})\to\ualgshift$ such that $\Xi([g])=\pi(g)$ for all $g\in\F$. Moreover, $\Xi$ restricts to an isomorphism between $A_{\operatorname{par}}(\F,\mathcal{R})$ and $\udalgshift$, which is equivariant with respect to the partial actions $\theta$ and $\tau$.
\end{theorem}

\begin{proof}
By Proposition~\ref{partial.rep}, Remark~\ref{firework}, and the definition of $R_{\operatorname{par}}(\F,\mathcal{R})$, there exists an $R$-algebra homomorphism $\Xi:R_{\operatorname{par}}(\F,\mathcal{R})\to\ualgshift$ such that $\Xi([g])=\pi(g)$ for all $g\in\F$, where $\pi$ is given in Proposition~\ref{partial.rep}. In particular, $\Xi(\varepsilon_{\beta\alpha^{-1}})=p_{C(\alpha,\beta)}$ for all $\alpha,\beta\in\lang$.

To define an inverse, for each $a\in\alf$ we map $s_a$ to $t_a:=[a]$ and $s_a^*$ to $t_a^*:=[a^{-1}]$. As for the elements $p_A$, where $A\in\TCB$, we use Remark~\ref{U_elements} and relation (i) in Definition~\ref{gelado} to write $p_A$ as
\[p_A=\sum_{i}\lambda_i\prod_{j}p_{C(\gamma_{ij},\beta_{ij})},\]
and map $p_A$ to
\[q_A:=\sum_{i}\lambda_i\prod_j\varepsilon_{\beta_{ij}\gamma_{ij}^{-1}}.\]
The relations in $\mathcal{R}$ ensure that two different ways of writing $p_A$ in the form above will yield the same element $q_A$ in $R_{\operatorname{par}}(\F,\mathcal{R})$. Moreover, using $\mathcal{R}$, we see that the family $\{q_A\}_{A\in\TCB}$ satisfies relation (i) of Definition~\ref{gelado}. That the family $\{t_a,t_a^*\}$ satisfies relation (ii) of Definition~\ref{gelado} follows from Lemma~\ref{semsentido}. For relation (iii) of Definition~\ref{gelado}, we first need a few preliminary computations. Let $\alpha=\alpha_1\ldots\alpha_n\in\lang$. By Remark~\ref{pausa},
\[p_{C(\eword,\alpha_1)}p_{C(\eword,\alpha_1\alpha_2)}\cdots p_{C(\eword,\alpha_1\ldots\alpha_n)}=p_{C(\eword,\alpha)},\]
so that
\[\varepsilon_{\alpha_1}\varepsilon_{\alpha_1\alpha_2}\cdots\varepsilon_{\alpha_1\ldots\alpha_{n-1}}\varepsilon_{\alpha}=\varepsilon_{\alpha}\]
in $R_{\operatorname{par}}(\F,\mathcal{R})$.

Thus, by Lemma~\ref{semsentido} and the relations in $\mathcal{R}$, we obtain
\begin{align*}
    [\alpha_1]\cdots[\alpha_n] &=\varepsilon_{\alpha_1}\varepsilon_{\alpha_1\alpha_2}\cdots\varepsilon_{\alpha_1\ldots\alpha_{n-1}}[\alpha_1\ldots\alpha_n]\\
    &=\varepsilon_{\alpha_1}\varepsilon_{\alpha_1\alpha_2}\cdots\varepsilon_{\alpha_1\ldots\alpha_{n-1}}\varepsilon_{\alpha}[\alpha] \\
    &=\varepsilon_{\alpha}[\alpha]\\
    &=[\alpha].
\end{align*}
Similarly, in $R_{\operatorname{par}}(\F,\mathcal{R})$, we have that $[\alpha_n^{-1}]\cdots[\alpha_1^{-1}]=[\alpha^{-1}]$. Let also $\beta=\beta_1\ldots\beta_m\in\lang$. Then
\begin{align*}
    t_{\beta_1}\cdots t_{\beta_m}t_{\alpha_n}^*\cdots t_{\alpha_1}^*t_{\alpha_1}&\cdots t_{\alpha_n}t_{\beta_m}^*\cdots t_{\beta_1}^* \\ 
    &=[\beta_1]\cdots[\beta_m][\alpha_n^{-1}]\cdots[\alpha_1^{-1}][\alpha_1]\cdots[\alpha_n][\beta_m^{-1}]\cdots[\beta_1^{-1}] \\
    &=[\beta][\alpha^{-1}][\alpha][\beta^{-1}] \\
    &=[\beta]\varepsilon_{\alpha^{-1}}[\beta^{-1}]\\
    &=\varepsilon_{\beta\alpha^{-1}}[\beta][\beta^{-1}]\\
    &=\varepsilon_{\beta\alpha^{-1}}\varepsilon_{\beta}\\
    &=q_{C(\alpha,\beta)}q_{C(\eword,\beta)}\\
    &=q_{C(\alpha,\beta)\cap C(\eword,\beta)}\\
    &=q_{C(\alpha,\beta)}.
\end{align*}
By the universal property of $\ualgshift$, we obtain an $R$-algebra homomorphism $\Lambda:\ualgshift\to R_{\operatorname{par}}(\F,\mathcal{R})$ such that $\Lambda(s_\alpha)=[\alpha]$, $\Lambda(s_\alpha^*)=[\alpha^{-1}]$, and $\Lambda(p_{C(\alpha,\beta)})=\varepsilon_{\beta\alpha^{-1}}$ for all $\alpha,\beta\in\lang$.

By the definitions of $\Lambda$ and $\Xi$, for every $a\in\alf$ and $A\in\TCB$, we have that $\Xi(\Lambda(s_a))=\Xi([a])=s_a$, $\Xi(\Lambda(s_a)^*)=\Xi([a^{-1}])=s_a^*$, and $\Xi(\Lambda(p_A))=\Xi(q_A)=p_A$.

Now, let $g\in\F$. If $g$ is simple with respect to $\osf$, say $g=\alpha\beta^{-1}$ in reduced form for some $\alpha,\beta\in\lang$, then
\begin{align*}
    \Lambda(\Xi([g]))&=\Lambda(s_\alpha s_\beta^*)\\
    &=[\alpha][\beta^{-1}]\\
    &=\varepsilon_{\alpha}[\alpha\beta^{-1}]\\
    &=\varepsilon_{\alpha}\varepsilon_{\alpha\beta^{-1}}[\alpha\beta^{-1}]\\
    &=\varepsilon_{\alpha\beta^{-1}}[\alpha\beta^{-1}]\\
    &=[\alpha\beta^{-1}]\\
    &=[g],
\end{align*}
where the fifth equality follows from the relations in $\mathcal{R}$, since $p_{C(\eword,\alpha)}p_{C(\beta,\alpha)}=p_{C(\beta,\alpha)}$. If $g$ is not simple with respect to $\osf$, then
$[g]=\varepsilon_g[g]=0$, and thus $\Lambda(\Xi([g]))=0=[g]$. This finishes the proof that $\Xi$ is an isomorphism. 

From the computations above, we see that the image of $A_{\operatorname{par}}(\F,\mathcal{R})$ under $\Xi$ is the subalgebra of $\ualgshift$ generated by projections of the form $p_{C(\alpha,\beta)}$. Due to relation (i) of Definition~\ref{gelado}, Remark~\ref{U_elements}, and Remark~\ref{pizza}, we have that the image of $A_{\operatorname{par}}(\F,\mathcal{R})$ is exactly $\udalgshift$.

To show that $\Xi$ is equivariant with respect to the partial actions $\theta$ and $\tau$, we must show that $\Xi\circ\theta_g(x) = \tau_g\circ\Xi(x)$, for every $g\in G$ and $x\in D_{g^{-1}}$. Since $D_{g^{-1}}$ is generated by $\{\varepsilon_{g^{-1}}\varepsilon_{h}\}_{h\in G}$, it is enough to verify the equality for $x = \varepsilon_{g^{-1}}\varepsilon_{h}$. We leave it to the reader to check that $$\Xi\circ\theta_g(\varepsilon_{g^{-1}}\varepsilon_{h}) = \tau_g\circ\Xi(\varepsilon_{g^{-1}}\varepsilon_{h})=0$$
if at least one of $g$, $h$ and $gh$ is not simple with respect to $\osf$. Now, consider the case where $g$, $h$, and $gh$ are simple with respect to $\osf$. As in Case 5 of the proof of Proposition~\ref{partial.rep}, if $g=\alpha_1\beta_1^{-1}$ and $h=\alpha_2\beta_2^{-1}$ are in reduced form, we must have $\alpha_2=\beta_1\gamma$ or $\beta_1=\alpha_2\gamma$. Assuming that $\alpha_2=\beta_1\gamma$ (the other case is analogous), then $gh = (\alpha_1\gamma)\beta_2^{-1}$ is the reduced form of $gh$. Therefore we have, on the one hand, that
$$\Xi\circ\theta_g(\varepsilon_{g^{-1}}\varepsilon_{h}) = \Xi(\varepsilon_{g}\varepsilon_{gh}) = p_{C(\beta_1,\alpha_1)}p_{C(\beta_2,\alpha_1\gamma)},$$
and, on the other hand, that
\begin{align*}
    \tau_g\circ\Xi(\varepsilon_{g^{-1}}\varepsilon_{h}) & = \tau_{\alpha_1\beta_1^{-1}}(p_{C(\alpha_1,\beta_1)}p_{C(\beta_2,\alpha_2)}) & \\
    & = \tau_{\alpha_1\beta_1^{-1}}(1_{C(\alpha_1,\beta_1)\cap C(\beta_2,\alpha_2)}) & \text{(see Remark~\ref{pizza})} \\
    & = 1_{C(\alpha_1,\beta_1)\cap C(\beta_2,\alpha_2)}\circ \widehat{\tau}_{\beta_1\alpha_1^{-1}} & \\
    & = 1_{\widehat{\tau}_{\alpha_1\beta_1^{-1}}(C(\alpha_1,\beta_1)\cap C(\beta_2,\alpha_2))} & \\
    & = 1_{C(\beta_1,\alpha_1)\cap C(\beta_2,\alpha_1\gamma)} & \text{(by \cite[Lemma 5.19(i)]{BCGWSubshift})} \\
    & = p_{C(\beta_1,\alpha_1)}p_{C(\beta_2,\alpha_1\gamma)}, &
\end{align*}
as desired.
\end{proof}

For each $x\in\osf$, let \[\xi_x=\{g\in \F:g=\beta\alpha^{-1}\text{ for some }\alpha,\beta\in\lang\text{ such that }x\in C(\alpha,\beta)\}.\]
In the next couple of lemmas, we compare the partial action $\widehat{\theta}$ of Theorem~\ref{grevedosalunos} with the partial action $\tauh$ described in this section.

\begin{lemma}\label{biscoito}
    Let $\alpha,\beta\in\lang$, $g\in\F$ and $x\in\osf$ be such that $\beta\alpha^{-1}$ is reduced in $\F$, $\beta\alpha^{-1}\in\xi_x$ and $\beta\alpha^{-1}g\in\xi_x$. Then $g$ is simple with respect to $\osf$.
\end{lemma}

\begin{proof}
    Since $\beta\alpha^{-1}\in \xi_x$, there exists $y\in\osf$ such that $x=\beta y$ and $\alpha y\in\osf$. By hypothesis, $\beta\alpha^{-1}g=\mu\nu^{-1}$ for some $\mu,\nu\in\lang$ such that $x=\mu z$ for some $z\in\osf$ with $\nu z\in\osf$. Since $\beta$ and $\mu$ are initial segments of $x$, we have that $\mu=\beta\mu'$ for some $\mu'\in\lang$, or $\beta=\mu\beta'$ for some $\beta'\in\lang$.

    Suppose first that $\mu=\beta\mu'$. In this case, $x=\mu z=\beta\mu' z$  so that $\mu' z=y$. Also, $g=\alpha\mu'\nu^{-1}$. Note that $\alpha \mu' z=\alpha y\in \osf$. Hence, $\alpha\mu'\in\lang$.

    Now suppose that $\beta=\mu\beta'$. In this case, $x=\mu\beta' y$ so that $z=\beta' y$. Also, $g=\alpha(\nu\beta')^{-1}$. Note that $\nu \beta' y=\nu z$ and hence $\nu\beta'\in\lang$.

    Since initial segments of elements in $\lang$ are still in $\lang$, $g$ is simple with respect to $\osf$.
\end{proof}

\begin{lemma}\label{equivariance}
    Let $\alpha,\beta\in\lang$ and $x\in\osf$ be such that $\beta\alpha^{-1}$ is in reduced form in $\F$ and $\xi_x\in\Omega_{\beta\alpha^{-1}}$. Then, $\widehat{\theta}_{\alpha\beta^{-1}}(\xi_x)=\xi_{\tauh_{\alpha\beta^{-1}}(x)}$.
\end{lemma}

\begin{proof}
     Since $\xi_x\in \Omega_{\beta\alpha^{-1}}$, we have $x\in C(\alpha,\beta)$. Thus, there exists $y\in \osf$ such that $x=\beta y$ and $\alpha y\in \osf$. 

     Let $g\in\F$. First, suppose $g\in\widehat{\theta}_{\alpha\beta^{-1}}(\xi_x)$. Then $\beta\alpha^{-1}g\in \xi_x$. By Lemma~\ref{biscoito}, $g$ is simple with respect to $\osf$. Let $\gamma,\delta\in\lang$ be such that $g=\gamma\delta^{-1}$ is in reduced form. Since $\beta\alpha^{-1}g\in \xi_x$, we have that $\beta\alpha^{-1}g=\beta\alpha^{-1}\gamma\delta^{-1}$ is simple with respect to $\osf$. Hence, $\alpha=\gamma\alpha'$ or $\gamma=\alpha\gamma'$.

     On the other hand, suppose $g\in\xi_{\tauh_{\alpha\beta^{-1}}(x)}$. Since $\tauh_{\alpha\beta^{-1}}(x) = \alpha y$, by the definition of $\xi_{\alpha y}$, there exists $\gamma,\delta\in\lang$ such that $g=\gamma\delta^{-1}$ and $\alpha y\in C(\delta,\gamma)$. We may assume that $g=\gamma\delta^{-1}$ is in reduced form. Since $\alpha y\in C(\delta,\gamma)$, it follows that $\alpha=\gamma\alpha'$ or $\gamma=\alpha\gamma'$.

     Thus, to show that $\widehat{\theta}_{\alpha\beta^{-1}}(\xi_x)=\xi_{\tauh_{\alpha\beta^{-1}}(x)}$, we only need to consider $g\in\F$ such that $g=\gamma\delta^{-1}$, with $\gamma,\delta\in\lang$, $g$ in reduced form, and $\alpha=\gamma\alpha'$ or $\gamma=\alpha\gamma'$. We consider three cases.
     
    Case 1: $\alpha=\gamma\alpha'$ with $|\alpha'|>0$. In this case, the reduced form of $\beta\alpha^{-1}\gamma\delta^{-1}$ is $\beta(\delta\alpha')^{-1}$ and $\alpha y=\gamma\alpha' y$. It follows that
    \begin{align*}
        g\in\widehat{\theta}_{\alpha\beta^{-1}}(\xi_x) &\iff\beta(\delta\alpha')^{-1}\in\xi_x \iff x\in C(\delta\alpha',\beta) \iff \delta\alpha' y\in \osf \\
        &\iff \alpha y\in C(\delta,\gamma)\iff g\in\xi_{\alpha y}=\xi_{\tauh_{\alpha\beta^{-1}}(x)}.
    \end{align*}

    Case 2: $\gamma=\alpha\gamma'$ with $|\gamma'|>0$. In this instance, the reduced form of $\beta\alpha^{-1}\gamma\delta^{-1}$ is $\beta\gamma'\delta^{-1}$, and $y=\gamma' z$ for some $z\in\osf$. Hence,
    \begin{align*}
        g\in\widehat{\theta}_{\alpha\beta^{-1}}(\xi_x)&\iff\beta\gamma'\delta^{-1}\in\xi_x\iff x\in C(\delta,\beta\gamma')\iff \delta z\in\osf \\
        & \iff \alpha y\in C(\delta,\gamma)\iff g\in\xi_{\tauh_{\alpha\beta^{-1}}(x)}.
    \end{align*}

    Case 3: $\gamma=\alpha$. In this case, the reduced form of $\beta\alpha^{-1}\gamma\delta^{-1}$ is $\beta'(\delta')^{-1}$, where $\beta=\beta'\mu$ and $\delta=\delta'\mu$ for some $\beta',\delta',\mu\in\lang$. Then,
    \begin{align*}
        g\in\widehat{\theta}_{\alpha\beta^{-1}}(\xi_x)& \iff\beta'(\delta')^{-1}\in \xi_x\iff x\in C(\delta',\beta')\iff \delta'\mu y \in \osf  \\ & \iff \alpha y\in C(\delta,\gamma)\iff g\in\xi_{\tauh_{\alpha\beta^{-1}}(x)}.
    \end{align*}

    These cases exhaust all possibilities, completing the proof.
\end{proof}

\begin{proposition}\label{homeomorphism}
    Under the conditions of Theorem~\ref{clouds}, suppose also that $R$ is indecomposable. Then the map $H:\HTCB\to \Omega_{\mathcal{R}}$ given by $H(\mathcal{F})=\{\alpha\beta^{-1}:C(\beta,\alpha)\in \mathcal{F}\}$ is a homeomorphism, where $\HTCB$ is as in Definition~\ref{diachuvoso}. Moreover, $H$ is equivariant with respect to $\widehat{\theta}$ and $\varphih$.
\end{proposition}

\begin{proof}
    By \cite[Proposition 5.20]{BCGWSubshift}, there is an isomorphism $\Upsilon:\udalgshift\to \Lc(\HTCB,R)$ such that $\Upsilon(p_A)=1_{O_A}$ for every $A\in\TCB$. Let $\rho:A_{\operatorname{par}}(\F)\to A_{\operatorname{par}}(\F,\mathcal{R})$ be the natural surjective homomorphism as in Section~\ref{greve}, let $\Theta^{-1}:\Lc(X_\F,R)\to A_{\operatorname{par}}(\F)$ as in Proposition~\ref{iso.theta}, let $\Xi$ as in Theorem~\ref{clouds}, and let $\Upsilon$ as above. Composing these maps, we obtain a surjective homomorphism $\phi:\Lc(X_\F,R)\to \Lc(\HTCB,R)$. 
    Note that $\ker(\phi)=\Theta(\ker\rho)$, so we can apply Lemma~\ref{bottle} to conclude that $\Omega_{\mathcal{R}}$ and $\HTCB$ are homeomorphic and that for every $\xi\in X_\F\setminus \Omega_{\mathcal{R}}$, there exists $f\in \Theta(\ker\rho)$ such that $f(\xi)=1$ (observe that the set $F$ in Lemma~\ref{bottle} is equal to $\Omega_{\mathcal{R}}$).
    
    We have already established that $\Omega_{\mathcal{R}}$ and $\HTCB$ are homeomorphic, but we need this homomorphism explicitly to show equivariance. Since Lemma~\ref{bottle} guarantees that the hypotheses of Theorem~\ref{grevedosalunos} are satisfied, it follows that there is an isomorphism $\Theta_{\mathcal{R}}:A_{\operatorname{par}}(\F,\mathcal{R})\to\Lc(\Omega_{\mathcal{R}},R)$. The isomorphism $\Xi:A_{\operatorname{par}}(\F,\mathcal{R})\to \udalgshift$ can then be translated as an isomorphism $\widehat{\Xi}:\Lc(\Omega_{\mathcal{R}},R)\to \Lc(\HTCB,R)$ in such a way that $\widehat{\Xi}(1_{\alpha\beta^{-1}})=1_{O_{C(\beta,\alpha)}}$, for every $\alpha,\beta\in\lang$ such that $\alpha\beta^{-1}$ is in reduced form in $\F$.

    Note that if $g\in F$ is not simple with respect to $\osf$, then for any element $\xi\in\Omega_{\mathcal{R}}$, we have that $g\notin \Omega_{\mathcal{R}}$. Indeed, by the definition of $\mathcal{R}$, we have that $1_g\in\Theta(\mathcal{R})$ so that $1_g(\xi)=0$.

    By \cite[Proposition~2.13]{BCGWSubshift}, there exists a homeomorphism $H:\HTCB\to \Omega_{\mathcal{R}}$ such that $\widehat{\Xi}(f)=f\circ H$ for all $f\in\Lc(\Omega_{\mathcal{R}})$. Given $\mathcal{F}\in\HTCB$ and $\alpha,\beta\in\lang$ such that $\alpha\beta^{-1}$ is in reduced form in $\F$, we have that
    \begin{align*}
        \alpha\beta^{-1}\in H(\mathcal{F}) &\iff 1_{\alpha\beta^{-1}}(H(\mathcal{F}))=1\\
        &\iff \widehat{\Xi}(1_{\alpha\beta^{-1}})=1  \\
        &\iff 1_{O_{C(\beta,\alpha)}}(\mathcal{F})=1 \\
        &\iff C(\beta,\alpha)\in\mathcal{F}\\
        &\iff \mathcal{F}\in O_{C(\beta,\alpha)}.
    \end{align*}
    Hence the map $H$ is given by $H(\mathcal{F})=\{\alpha\beta^{-1}:C(\beta,\alpha)\in \mathcal{F}\}$.

    For $x\in\osf$ and $\mathcal{F}_x$ as in Proposition~\ref{iotamap}, observe that $H(\mathcal{F}_x)=\xi_x$. To prove the equivariance of $H$, it is enough to consider $g=\alpha\beta^{-1}$, where $\alpha,\beta\in\lang$ are such that $\alpha\beta^{-1}$ is in reduced form in $\F$. Suppose that $x\in C(\alpha,\beta)$. By Proposition~\ref{iotamap} and Lemma~\ref{equivariance}, we have that
    \[\widehat{\theta}_{\alpha\beta^{-1}}(H(\mathcal{F}_x))=\widehat{\theta}_{\alpha\beta^{-1}}(\xi_x)=\xi_{\tauh_{\alpha\beta^{-1}(x)}}=H(\mathcal{F}_{\tauh_{\alpha\beta^{-1}(x)}})=H(\varphih_{\alpha\beta^{-1}}(\mathcal{F}_x)).\]
    By Proposition~\ref{iotamap}, the set $\{\mathcal{F}_x:x\in\osf\}$ is dense in $\HTCB$. Since all the domains of the partial action are open and the maps are homeomorphisms, we see that $H$ is equivariant with respect to $\widehat{\theta}$ and $\varphih$.
\end{proof}

\subsection{Simplicity of subshift algebras}

Aiming to describe simplicity of the subshift algebra, we now use Proposition~\ref{homeomorphism} to describe the minimality of the partial action $\varphih$. More precisely, we describe the minimality of $\widehat{\theta}$. For the reader's convenience, we recall below the concepts used in \cite{MishaRuy} for subshifts over finite alphabets.

\begin{definition}\cite[Definitions~13.10 and 13.11]{MishaRuy}
    Given $B\scj\lang$, the \emph{follower set of} $B$ is the set $F_B=\cap_{\beta\in B}F_\beta$. We say that $\osf$ is \emph{collectively cofinal} if for every $B\scj \lang$ such that $F_B\neq\emptyset$ and every $x\in \osf$, there are $\alpha,\gamma\in\lang$ such that $x\in C(\beta\gamma,\alpha)$ for all $\beta\in B$.
\end{definition}

\begin{definition}\cite[Definitions~13.14 and 13.15]{MishaRuy}
    Given $x\in\osf$ and $\beta\in\lang$, the \emph{cost} of reaching $x$ from $\beta$ is
    \[\operatorname{Cost}(\beta,x)=\inf\{|\alpha|+|\gamma|:\alpha,\gamma\in\lang\text{ and }x\in C(\beta\gamma,\alpha)\},\]
    with the convention that $\inf\emptyset=\infty$. We say that $\osf$ is \emph{strongly cofinal} if $\sup_{x\in\osf}\operatorname{Cost}(\beta,x)<\infty$ for all $\beta\in\lang$.
\end{definition}

\begin{definition}\cite[Definition~13.20]{MishaRuy}
    Given $x\in\osf$ and $B$ a finite subset of $\lang$, the \emph{cost} of reaching $x$ from $B$ is
    \[\operatorname{Cost}(B,x)=\inf\{|\alpha|+|\gamma|:\alpha,\gamma\in\lang\text{ and }x\in C(\beta\gamma,\alpha)\text{ for all }\beta\in B\},\]
    with the convention that $\inf\emptyset=\infty$. We say that $\osf$ is \emph{hyper cofinal} if $\sup_{x\in\osf}\operatorname{Cost}(B,x)<\infty$ for all finite $B\scj\lang$ such that $F_B\neq\emptyset$. 
\end{definition}

\begin{proposition}\cite[Proposition~13.21]{MishaRuy}
    Let $\osf$ be a subshift over a finite alphabet. The following are equivalent:
    \begin{enumerate}[(i)]
        \item $\osf$ is both collectively cofinal and strongly cofinal,
        \item $\osf$ is hyper cofinal,
        \item the partial action $\widehat{\theta}$ is minimal.
    \end{enumerate}
\end{proposition}
    
Their proof heavily relies on the finitude of the alphabet. We propose two new notions of confinality that are more adapted to work with infinite alphabets.

\begin{definition}
    We say that $\osf$ is \emph{sequentially cofinal} if for every finite subset $B$ of $\lang$ such that $F_B\neq\emptyset$ and every sequence $(x_n)_{n\in\nn}$ of elements in $\osf$, there are $\alpha,\gamma\in\lang$ and a subsequence $(x_{n_k})_{k\in\nn}$ such that $x_{n_k}\in C(\beta\gamma,\alpha)$ for all $k\in\nn$ and all $\beta\in B$.
    We say that $\osf$ is \emph{compactly cofinal} if for every $B$ finite subset of $\lang$ such that $F_B\neq\emptyset$ and every net $(x_i)_{i\in I}$ of elements in $\osf$, there are $\alpha,\gamma\in\lang$ and a subnet $(x_{i_j})_{j\in J}$ such that $x_{i_j}\in C(\beta\gamma,\alpha)$ for all $j\in J$ and all $\beta\in B$.
\end{definition}

Before describing the minimality of $\widehat{\theta}$, we need the following observation and a couple of lemmas. From the definition of partial actions, it follows that
\begin{equation}\label{epica}
    \widehat{\theta}_{(\beta\gamma)^{-1}}(\Omega_{\beta\alpha^{-1}}\cap\Omega_{\beta\gamma})=\Omega_{(\beta\gamma)^{-1}}\cap \Omega_{(\alpha\gamma)^{-1}}
\end{equation}
for all $\alpha,\beta,\gamma\in\lang$.

\begin{lemma}\label{titan}
    For each $x\in\osf$, we have that $\xi_x\in\Omega_{\mathcal{R}}$. Moreover, the set $\{\xi_x\}_{x\in\osf}$ is dense in $\Omega_{\mathcal{R}}$.
\end{lemma}

\begin{proof}
    Note that $\xi_x=H(\mathcal{F}_x)$ for each $x\in\osf$. The result then follows from \cite[Proposition~5.17]{BCGWSubshift} and Proposition~\ref{homeomorphism}.
\end{proof}

\begin{lemma}\label{orbit}
    For every non-empty open set $U\scj\Omega_{\mathcal{R}}$, there exist $B\scj\lang$ finite and $g\in G$ such that $\emptyset\neq \bigcap_{\beta\in B}\Omega_{\beta^{-1}}\scj \widehat{\theta}_g (U\cap \Omega_{g^-1})$.
\end{lemma}

\begin{proof}
    Since $\Omega_g=\emptyset$ if $g$ is not simple with respect to $\osf$, there exists a non-empty $V$ of the form
    \[V=\Omega_{\alpha_1\beta_1^{-1}}\cap\cdots\cap\Omega_{\alpha_n\beta_n^{-1}}\cap \Omega_{\mu_1\nu_1^{-1}}^c\cap\cdots\cap\Omega_{\mu_m\nu_m^{-1}}^c\]
    such that $V\scj U$. By Lemma~\ref{titan}, there exists $x\in\osf$ such that $\xi_x\in V$.
    
    Let $k>\max\{|\alpha_1|,\ldots,|\alpha_n|,|\mu_1|,\ldots,|\mu_m|\}$ and $g_1=x_0\ldots x_{k-1}$. Note that each $\alpha_i$ is a beginning of $g_1$, and for each $\mu_j$, which is not a beginning of $g_1$, the element of the free group $g_1^{-1}\mu_j\nu_j^{-1}$ is not simple with respect to $\osf$. By \eqref{epica}, we have that
    \[\widehat{\theta}_{g_1^{-1}}(\xi_x)\in \widehat{\theta}_{g_1^{-1}}(V \cap \Omega_{g_1})=\Omega_{g_1^{-1}}\cap \Omega_{\gamma_1^{-1}}\cap\cdots\cap\Omega_{\gamma_n^{-1}}\cap\Omega_{\delta_1^{-1}}^c\cap\cdots\cap\Omega_{\delta_l^{-1}}^c,\]
    for some $\gamma_1,\ldots,\gamma_n,\delta_1,\ldots,\delta_l\in\lang$.
    Write $x=g_1y$ and note that for each $j=1,\ldots,l$, $y\notin C(\delta_j,\eword)$. Hence, there exists $k'$ such that $\delta_jy_0\ldots y_{k'-1}\notin\lang$ for all $j\in\{1,\ldots,l\}$. Now, let $g_2=y_0\ldots y_{k'-1}$. Because $C(g_1g_2,\eword)\scj C(g_2,\eword)$, we have that $\Omega_{(g_1g_2)^{-1}}\scj \Omega_{g_2^{-1}}$. Hence
    \[\widehat{\theta}_{g_2^{-1}}(\xi_y)\in \Omega_{(g_1g_2)^{-1}}\cap\Omega_{(\gamma_1g_2)^{-1}}\cap\cdots\cap\Omega_{(\gamma_ng_2)^{-1}}= \widehat{\theta}_{(g_1g_2)^{-1}}(V \cap \Omega_{g_1g_2})\scj \widehat{\theta}_{(g_1g_2)^{-1}}(U \cap \Omega_{g_1g_2}).\]
    Thus, the result holds for $g=(g_1g_2)^{-1}$ and $B=\{g_1g_2,\gamma_1g_2,\ldots,\gamma_ng_2\}$.
\end{proof}

\begin{theorem}\label{minimal}
    The partial action $\widehat{\theta}$ is minimal if and only if $\osf$ is compactly cofinal. Moreover, if the alphabet is countable, then $\widehat{\theta}$ is minimal if and only if $\osf$ is sequentially cofinal.
\end{theorem}

\begin{proof}
    ($\Rightarrow$) Suppose $\widehat{\theta}$ is minimal. Let $B\scj\lang$ be finite such that $F_B\neq\emptyset$ and let $(x_i)_{i\in I}$ a net in $\osf$. Since $\HTCB$ is compact, there exists a subnet $(x_{i_j})_{j\in J}$ such that $\mathcal{F}_{x_{i_j}}$ converges to some $\mathcal{F}\in\HTCB$. By minimality, there are $\alpha,\gamma\in\lang$ such that $C(\gamma,\alpha)\in\mathcal{F}$ and $\varphih_{\gamma\alpha^{-1}}(\mathcal{F})\in O_{F_B}$. This implies that $C(\beta\gamma,\alpha)\in\mathcal{F}$ for all $\beta\in B$. Since $O_A$ is open for every $A\in\TCB$, there exists $j_0\in J$ such that $x_{i_j}\in C(\beta\gamma,\alpha)$ for every $j\geq j_0$ and every $\beta\in B$. Considering $J' = \{j\in J: j\geq j_0\}$, $(x_{i_j})_{j\in J'}$ is a subnet of $(x_i)_{i\in I}$ such that $x_{i_j}\in C(\beta\gamma,\alpha)$ for every $j\in J'$, showing that $\osf$ is compactly cofinal.

    ($\Leftarrow$) Suppose $\osf$ is compactly cofinal. By Lemma~\ref{orbit}, it is enough to prove that for each finite $B\scj\lang$ such that $U:=\bigcap_{\beta\in B} \Omega_{\beta^{-1}}$ is non-empty, there exists $\xi\in\Omega_{\mathcal{R}}$ such that $\Orb(\xi)\cap U\neq \emptyset$. It then enough to find $\alpha,\gamma\in\lang$ such that $\alpha\gamma^{-1}\in \xi$ and $\alpha\gamma^{-1}\beta^{-1}\in \xi$ for all $\beta\in B$. Since $\Omega_{\mathcal{R}}$ is compact and, by Lemma~\ref{titan}, $\{\xi_x : x\in\osf\}$ is dense in $\Omega_{\mathcal{R}}$, there exists a net $(x_i)_{i\in I}$ in $\osf$ such that $\xi_{x_i}$ converges to $\xi$. By hypothesis, there are $\alpha,\gamma\in\lang$ and a subnet $(x_{i_j})_{j\in J}$ such that $x_{i_j}\in C(\beta\gamma,\alpha)$ for all $j\in J$ and all $\beta\in B$. In particular, we also have that $x_{i_j}\in C(\gamma,\alpha)$ for all $j\in J$. Since for each $g\in \F$, we have that $\Omega_g$ is closed in $\Omega_{\mathcal{R}}$, we have that $\alpha\gamma^{-1}\in \xi$ and $\alpha\gamma^{-1}\beta^{-1}\in \xi$ for all $\beta\in B$, as required.

    If the alphabet is countable, then $\HTCB$ is a second countable compact Hausdorff space, and so it is metrizable. In this case, we can adapt the proof by replacing nets with sequences and subnets with subsequences.
\end{proof}

We can now characterize when the subshift algebra $\ualgshift$ is simple. For that, we need to recall a definition.

\begin{definition}\cite[Definition~6.2]{BCGWCStarSubshift}
    We say that a subshift $\osf$ satisfies condition (L) if for every finite set $B\scj\lang$, and every $\gamma\in\lang$ such that $\gamma^{\infty}\in F_B$, there exists an element in $F_B$ other than $\gamma^{\infty}$.
\end{definition}

We then obtain the following.

\begin{theorem}\label{simple}
   Let $R$ be a field, and $\osf$ a subshift over an arbitrary alphabet. The algebra $\ualgshift$ is simple if and only if $\osf$ satisfies condition (L) and is compactly cofinal.
\end{theorem}

\begin{proof}
    By Theorem~\ref{clouds}, we have an isomorphism between $R_{\operatorname{par}}(\F,\mathcal{R})$ and $\ualgshift$. By Corollary~\ref{SteinbergAgain}, $\ualgshift$ is isomorphic to the Steinberg algebra of the transformation group\-oid $\F\ltimes_{\widehat{\theta}}\Omega_{\mathcal{R}}$. Now, we can use \cite[Theorem~4.1]{SimplicityGroupoid} to characterize the simplicity of $\ualgshift$. The condition about the groupoid being effective is given by \cite[Theorem~6.5]{BCGWCStarSubshift} and the condition about minimality is given by Theorem~\ref{minimal}.
\end{proof}

\begin{remark}
In Theorem~\ref{simple}, if the alphabet is countable, compactly cofinal can be replaced with sequentially cofinal. If the alphabet is finite, compactly cofinal can instead be replaced with hyper cofinal.  Furthermore, Theorem~\ref{simple} can also be deduced from \cite[Theorem 5.22]{BGMJ}, which additionally offers a characterization of the relationship between the primeness of a partial skew groupoid ring and the topological transitivity of the action.
\end{remark}

For a subshift $\osf$, let $\ucalgshift$ denote the unital subshift C*-algebra defined in \cite{BCGWCStarSubshift}.
\begin{theorem}\label{simplecstar}
    Let $\osf$ be a subshift over a countable alphabet $\alf$. The C*-algebra $\ucalgshift$ is simple if and only if $\osf$ satisfies condition (L) and is sequentially cofinal.
\end{theorem}

\begin{proof}
    By \cite[Theorem 6.1]{BCGWCStarSubshift}, $\ucalgshift$ is isomorphic to the groupoid C*-algebra of the groupoid $\F\ltimes_{\varphih}\HTCB$. By \cite[Theorem~6.4]{BCGWSubshift}, the groupoid $\F\ltimes_{\varphih}\HTCB$ is isomorphic to a Deaconu-Renault groupoid. Since the alphabet is countable, this groupoid is second countable, so we can apply \cite[Proposition~2.4]{Renault00} to conclude that $\F\ltimes_{\varphih}\HTCB$ is amenable. By Proposition~\ref{homeomorphism}, the groupoids $\F\ltimes_{\varphih}\HTCB$ and $\F\ltimes_{\widehat{\theta}}\Omega_{\mathcal{R}}$ are isomorphic. Now, as in the proof of Theorem~\ref{simple}, we can use \cite[Theorem~5.1]{SimplicityGroupoid} to characterize the simplicity of $\ucalgshift$.
\end{proof}

\begin{remark}
    For a subshift over a finite alphabet, we can replace sequentially cofinal with hyper cofinal in Theorem~\ref{simplecstar} (cf. \cite[Theorem~14.5]{MishaRuy} and \cite[Theorem~7.13]{GillesDanie}).
\end{remark}

%
%
%
%
%

\bibliographystyle{abbrv}
\bibliography{BCGWPartialGroupAlgebras}

\end{document}